\documentclass[12pt,twoside,reqno]{amsart}
\usepackage{amsmath}
\usepackage{amsfonts}
\usepackage{amssymb}
\usepackage{color}
\usepackage{mathrsfs}
\newcommand{\RR}{\mathbb R}

\newcommand{\ml}{\mathcal{L}}
\newcommand{\mn}{\mathcal{N}}
\newcommand{\mb}{\mathcal{B}}
\newcommand{\mc}{\mathcal{C}}
\newcommand{\mj}{\mathcal{J}}
\newcommand{\bqnn}{\begin{equation*}}
\newcommand{\eqnn}{\end{equation*}}
\newcommand{\bqn}{\begin{equation}}
\newcommand{\eqn}{\end{equation}}

\DeclareMathAlphabet{\mathpzc}{OT1}{pzc}{m}{it}
\newtheorem{theorem}{Theorem}[section]

\newtheorem{lemma}[theorem]{Lemma}
\newtheorem{proposition}[theorem]{Proposition}

\newtheorem{remark}[theorem]{Remark}
\numberwithin{equation}{section}
\newcommand{\e}{\varepsilon}
\begin{document}
\title{Global weak solutions for a degenerate parabolic system modeling the spreading of insoluble surfactant} 
\thanks{This work was partially supported by the french-german PROCOPE project 20190SE}
\author{Joachim Escher}
\address{Leibniz Universit\"at Hannover, Institut f\"ur Angewandte Mathematik, Welfengarten 1, D--30167 Hannover, Germany}
\email{escher@ifam.uni-hannover.de}
\author{Matthieu Hillairet}
\address{Institut de Math\'{e}matiques de Toulouse, CNRS UMR~5219, Universit\'{e} de Toulouse, F--31062 Toulouse cedex 9, France}
\email{matthieu.hillairet@math.univ-toulouse.fr}
\author{Philippe Lauren\c cot}
\address{Institut de Math\'ematiques de Toulouse, CNRS UMR~5219, Universit\'e de Toulouse, F--31062 Toulouse cedex 9, France} 
\email{laurenco@math.univ-toulouse.fr}
\author{Christoph Walker}
\address{Leibniz Universit\"at Hannover, Institut f\"ur Angewandte Mathematik, Welfengarten 1, D--30167 Hannover, Germany}
\email{walker@ifam.uni-hannover.de}


\date{\today}

%%%%%%%%%%%%%%%%%%%%%%%%%%%%%%%%%%%%%%%%%%%%%%%%%%%%%%%%%%%%%%%%%%%%
\begin{abstract}
We prove global existence of a nonnegative weak solution to a degenerate parabolic system, which models the spreading of insoluble surfactant on a thin liquid film.
\end{abstract}
%%%%%%%%%%%%%%%%%%%%%%%%%%%%%%%%%%%%%%%%%%%%%%%%%%%%%%%%%%%%%%%%%%%%

\maketitle

%
%     HEADLINES
%
\pagestyle{myheadings}
\markboth{\sc{J.~Escher, M.~Hillairet, Ph.~Lauren\c cot, and Ch.~Walker}}{\sc{A degenerate parabolic system modeling the spreading of insoluble surfactant}}

%%%%%%%%%%%%%%%%%%%%%%%%%%%%%%%%%%%%%%%%%%%%%%%%%%%%%%%%%%%%%%%%%%%%
%%%%%%%%%%%%%%%%%%%%%%%%%%%%%%%%%%%%%%%%%%%%%%%%%%%%%%%%%%%%%%%%%%%%
\section{Introduction}\label{sec:int}
%%%%%%%%%%%%%%%%%%%%%%%%%%%%%%%%%%%%%%%%%%%%%%%%%%%%%%%%%%%%%%%%%%%%
%%%%%%%%%%%%%%%%%%%%%%%%%%%%%%%%%%%%%%%%%%%%%%%%%%%%%%%%%%%%%%%%%%%%

It is a widely used approach in the study of the dynamical behavior of viscous thin films to approximate
the full fluid mechanical system by simpler model equations, using e.g. lubrication theory  and cross-sectional averaging. 
In most of such models surface tension effects may then become significant, or even dominant. Therefore, also the influence of surfactant, i.e. surface active agents on the free surface of thin films, is of considerable importance. A surfactant lowers the surface tension of the liquid and the resulting gradients of surface tension induce so-called Marangoni stresses which in turn cause a spreading of the surfactant on the interface. We investigate here a model in which the surfactant is assumed to be insoluble. In addition we include gravity but neglect effects of capillarity and van der Waals forces. Writing $h(t,x)$ for the film thickness and $\Gamma(t,x)$ for the concentration of surfactant at time $t>0$ and position $x\in (0,L)$,
Jensen and Grotberg derived in \cite{JeGr92, JeGr93} the following system:
\begin{eqnarray}
\partial_t h & =& \partial_x \left( \frac{G h^3}{3}\ \partial_x h - \frac{h^2}{2}\ \partial_x \sigma(\Gamma) \right) \;\;\mbox{ in }\;\;  Q_\infty\,, \label{a1} \\
\partial_t \Gamma & = & \partial_x \left( \frac{G h^2}{2}\ \Gamma\ \partial_x h + \left( D - h\ \Gamma\ \sigma'(\Gamma) \right)\ \partial_x \Gamma \right) \;\;\mbox{ in }\;\; Q_\infty\,. \label{a2} 
\end{eqnarray}
Here $Q_\infty:= (0,\infty)\times (0,L)$ denotes the time-space domain of the unknowns $h$ and $\Gamma$, with $L$ being the spatial horizontal latitude of the system. We further impose no-flux boundary condition for $h$ and $\Gamma$, i.e.
\begin{equation}
\partial_x h  =  \partial_x \Gamma = 0\;\;\mbox{ on }\;\; (0,\infty)\times \{0,L\}\,, \label{a3}
\end{equation}
as well as initial conditions for these quantities:
\begin{equation}
(h,\Gamma)(0)  =  (h_0,\Gamma_0) \;\;\mbox{ in }\;\; (0,L)\,, \label{a4}
\end{equation}
where $h_0$ and $\Gamma_0$ are given. Equation (\ref{a1}) for the height function $h$ is a consequence of the conservation of momentum and the kinematic boundary condition, reflecting the model assumption that the velocity of the free interface balances the normal component of the liquid, cf. \cite{EHLW10, JeGr92, JeGr93}. Equation (\ref{a2}) is an advection-transport equation for the surfactant concentration on the interface in which $D$ is a non-dimensional surface diffusion coefficient, assumed to be positive and constant. The positive constant $G$ represents a gravitational force.

Of considerable importance in the modeling is the surface tension $\sigma(\Gamma)$, a decreasing function of the surfactant concentration. Several equations of state giving the dependence of the surface tension $\sigma$ upon the surfactant concentration $\Gamma$, including
$$
\sigma(\Gamma) = \sigma_s - \beta\ \Gamma \;\;\mbox{ or }\;\; \sigma(\Gamma) = \sigma_s - \beta\ \ln{\left( 1 \pm \frac{\Gamma}{\Gamma_\infty}\right)}\,,
$$ 
may be found in the literature, see \cite{CF95,JeGr92,MSL93} and the references therein. In this paper, for technical reasons we assume that 
\begin{equation}
\sigma\in \mathcal{C}^3([0,\infty))\,, \qquad \sigma(0)>0\,, \qquad 0<\sigma_0 \le - \sigma' \le \sigma_\infty\,, \label{a6}
\end{equation}
which is satisfied in particular by the first example above. A straightforward consequence of \eqref{a6} is the fact that $\sigma$ grows at most linearly:
\begin{equation}
|\sigma(r)| \le \sigma(0) + \sigma_\infty\ r\,, \quad r\ge 0\,.\label{a6b}
\end{equation}

Observe that the coupled system (\ref{a1}),\;(\ref{a2}) is degenerate parabolic in the sense that parabolicity is lost if $h$ or $\Gamma$ vanish. While modeling issues related to surfactant spreading on thin liquid films have attracted considerable interest (e.g., see \cite{deWitGallez, JeGr92, JeGr93,Matar02} and the references therein), much less research has been dedicated to analytical aspects. In \cite{Renardy96a, Renardy96b, Renardy97} local existence results are shown. In \cite{GarckeWieland} global existence of weak solutions is derived for a variant of (\ref{a1}),\;(\ref{a2}) without gravity but including a fourth order term in $h$ modeling capillarity effects. Local asymptotic stability of steady states (being simply the positive constants) is investigated in \cite{EHLW10} for the case of soluble surfactant. These results in particular show that, starting with initial values near steady states, problem (\ref{a1})-(\ref{a4}) admits a unique global positive classical solution.

Our aim here is to prove the existence of global nonnegative weak solutions to (\ref{a1})-(\ref{a4}) for arbitrary nonnegative initial values. The core of our analysis is the fact that system (\ref{a1})-(\ref{a3}) possesses an energy functional entailing various {\it a priori} estimates on $(h,\Gamma)$. We regularize (\ref{a1})-(\ref{a4}) appropriately  to obtain a uniformly parabolic system with coefficients (depending nonlinearly on $(h,\Gamma)$) being regular enough to apply abstract semi-group theory to prove well-posedness of the regularized system. This approach warrants that the thereby constructed nonnegative solutions exist globally provided they are {\it a priori} bounded in $W^1_2$. The aforementioned energy estimates provide such bounds and inherit also compactness properties in suitable function spaces to the family of regularized solutions, which allows us to extract a subsequence converging to a weak solution. \\

In fact, we shall establish the following result:

\begin{theorem}\label{MR} Let $D,\ G>0$ and suppose \eqref{a6}. Given nonnegative $h_0, \Gamma_0\in W_2^2(0,L)$ satisfying $\partial_x h_0(x)=\partial_x \Gamma_0(x)$ at $x=0$ and $x=L$, there exists a global weak solution to \eqref{a1}-\eqref{a4}, i.e. a pair of nonnegative functions $(h,\Gamma)$ such that $h(0)=h_0$, $\Gamma(0)=\Gamma_0$,
\begin{eqnarray*}
& & h\in L_\infty(0,T;L_2(0,L))\cap L_5(0,T;\mathcal{C}^{1/5}([0,L]))\,, \quad h^{5/2}\in L_2(0,T;W_2^1(0,L)\,, \\
& & \Gamma \in L_\infty(0,T;L_1(0,L))\cap L_2(0,T;\mathcal{C}([0,L]))\,, \quad \partial_x \sigma(\Gamma) \in L_1((0,T)\times (0,L))\,, \\
& & j_f := \left( - \frac{2}{5}\sqrt{\frac{G}{3}}\ \partial_x \left( h^{5/2} \right) + \sqrt{\frac{3h}{4G}}\ \partial_x \sigma(\Gamma) \right) \in L_2((0,T)\times (0,L))\,, \\
& & j_s := - \frac{G}{5}\ \partial_x \left( h^{5/2} \right) + \sqrt{h}\ \partial_x \sigma(\Gamma) \in L_2((0,T)\times (0,L))\,,
\end{eqnarray*}
for all $T>0$, and 
\begin{eqnarray*}
\frac{d}{dt} \int_0^L h\ \psi\ dx & = & \sqrt{\frac{G}{3}}\int_0^L \partial_x \psi\ h^{3/2}\ j_f\ dx\,, \\
\frac{d}{dt} \int_0^L \Gamma\ \psi\ dx & = & \int_0^L \partial_x\psi\ \left( -D\ \partial_x\Gamma + \sqrt{h}\ \Gamma\ j_s \right)\ dx \,,
\end{eqnarray*}
for all $\psi\in W_\infty^1(0,L)$. In addition, introducing the function $\phi$ defined by $\phi''(r)=-\sigma'(r)/r$ for $r>0$ and $\phi(1)=\phi'(1)=0$, the weak solution $(h,\Gamma)$ satisfies 
\begin{eqnarray}
& & \|h(t)\|_1 = \|h_0\|_1\,, \qquad \|\Gamma(t)\|_1 = \|\Gamma_0\|_1\,, \qquad t\ge 0\,, \label{consmatter} \\
& & \mathcal{L}_0(t) + \int_0^t \mathcal{D}_0(s)\ ds \le \mathcal{L}_0(0)\,, \qquad t\ge 0\,, \label{energineq}
\end{eqnarray}
with
$$
\mathcal{L}_0(t) := \int_0^L \left[ \frac{G}{2}\ |h(t,x)|^2 + \phi(\Gamma(t,x)) \right]\ dx\,, 
$$
and
\begin{eqnarray*}
2\mathcal{D}_0(t) & := & G\ \|j_f(t)\|_2^2 + \|j_s(t)\|_2^2 + \frac{G^2}{75}\ \left\|\partial_x (h^{5/2})(t) \right\|_2^2 \\
& & \ + \frac{1}{4}\ \left\| \sqrt{h(t)}\ \partial_x\sigma(\Gamma)(t) \right\|_2^2 + 8\sigma_0 D\ \left\| \partial_x \sqrt{\Gamma(t)} \right\|_2^2 \,.
\end{eqnarray*}

\end{theorem}

Observe that the notion of weak solutions is readily obtained by testing (\ref{a1})--(\ref{a4}) against $\psi\in W^1_\infty(0,L)$ and integrating with respect to the spatial variable. Also, the weak formulation ensures the time continuity of $h$ and $\Gamma$ (for some suitable weak topology with respect to space), so that the initial data $(h,\Gamma)(0)=(h_0,\Gamma_0)$ are meaningful. Observe finally that, thanks to the ``energy inequality'' \eqref{energineq}, we actually have an improved regularity on $\Gamma$, namely $\sqrt{\Gamma}\in L_2(0,T;W_2^1(0,L))$.

%%%%%%%%%%%%%%%%%%%%%%%%%%%%%%%%%%%%%%%%%%%%%%%%%%%%%%%%%%%%%%%%%%%%
\begin{remark}
Theorem~\ref{MR} is actually valid under the weaker assumption $h_0, \Gamma_0\in W_2^1(0,L)$, the proof being similar to the one given below using an additional (but classical) regularization of the initial data.
\end{remark}
%%%%%%%%%%%%%%%%%%%%%%%%%%%%%%%%%%%%%%%%%%%%%%%%%%%%%%%%%%%%%%%%%%%%

We close the introduction by outlining the content of our paper. Section 2 is devoted to a regularized version of the original system (\ref{a1})--(\ref{a4}). Roughly speaking, the coupling terms in \eqref{a1}-\eqref{a2} being of the same order as the diagonal terms (so that we are dealing with somehow a full diffusion matrix), we have to mollify them in order to be able to apply the abstract theory  developed in \cite{AmannTeubner93} for quasilinear parabolic systems. The regularization is then crucial to establish global existence. First, we derive the local well-posedness and {\it a priori} estimates then ensure the global well-posedness. Various compactness properties of the family of regularized solutions are established in Section 3, allowing us to recover a weak solution in the sense of Theorem~\ref{MR}. In the Appendix we collect some tools used for the compactness arguments in Section~3.

%%%%%%%%%%%%%%%%%%%%%%%%%%%%%%%%%%%%%%%%%%%%%%%%%%%%%%%%%%%%%%%%%%%%
%%%%%%%%%%%%%%%%%%%%%%%%%%%%%%%%%%%%%%%%%%%%%%%%%%%%%%%%%%%%%%%%%%%%
\section{A regularized problem}\label{sec:arp}
%%%%%%%%%%%%%%%%%%%%%%%%%%%%%%%%%%%%%%%%%%%%%%%%%%%%%%%%%%%%%%%%%%%%
%%%%%%%%%%%%%%%%%%%%%%%%%%%%%%%%%%%%%%%%%%%%%%%%%%%%%%%%%%%%%%%%%%%%

For $\e\in (0,1)$ and $f\in L_2(0,L)$, let $\mathcal{N}_\e(f)$ be the unique solution in $W_2^2(0,L)$ to the elliptic boundary-value problem
\begin{equation}
\mathcal{N}_\e(f) - \e^2\ \partial_x^2 \mathcal{N}_\e(f) = f \;\;\mbox{ in }\;\; (0,L)\,, \quad \partial_x \mathcal{N}_\e(f)(0)=\partial_x\mathcal{N}_\e(f)(L)=0\,. \label{b0}
\end{equation}
Clearly, $$\mathcal{N}_\e\in\ml(L_2(0,L),W_2^2(0,L))\cap \ml(\mc^\gamma([0,L]),\mc^{\gamma+2}([0,L]))\ ,\quad\gamma>0\ ,$$ is a positive operator.
We define the following functions
\begin{equation}
a_1(r) := G\ \frac{r^3}{3}\,, \quad a_{2,\e}(r) := \frac{(r-\sqrt{\e})^2}{2}\,,\quad r \ge 0\,, \label{b5}
\end{equation}
\begin{equation}
b_{2,\e}(r) := r-\e\,, \quad r\ge 0\,, \label{b6}
\end{equation}
and notice that
\begin{equation}
G\ a_{2,\e}^2(r) \le \eta\ r\ a_1(r)\,, \quad  r\ge \sqrt{\e}\,, \;\;\mbox{ with }\;\; \eta := \frac{3}{4}\,.\label{b14}
\end{equation}
We next fix $\eta_1\in (\eta,1)$ and define
\begin{equation}
\alpha_0(r,s) := \eta_1\ s + (1-\eta_1)\ r\,, \quad \alpha_1(r) := \int_0^r \sqrt{a_1(\rho)}\ d\rho\,, \quad r\ge 0\,, \ s\ge 0\,, \label{b7}
\end{equation}
and
\begin{equation}
\beta_1(r) := \int_0^r \rho\ |\sigma'(\rho)|\ d\rho\,, \quad r\ge 0\,.\label{b8}
\end{equation}

The regularized problem reads
\begin{eqnarray}
\partial_t h_\e & = & \partial_x \left( a_1(h_\e)\ \partial_x h_\e - \frac{a_{2,\e}(h_\e) \sqrt{\mn_\e(h_\e)}}{\sqrt{h_\e}}\ \partial_x \Sigma_\e(h_\e,\Gamma_\e) \right) \;\;\mbox{ in }\;\; Q_\infty\,, \label{b1aa} \\
\partial_t \Gamma_\e & = & \partial_x \left( G\ \frac{a_{2,\e}(h_\e) b_{2,\e}(\Gamma_\e) \sqrt{\alpha_0(h_\e,\mn_\e(h_\e))}}{\sqrt{h_\e a_1(h_\e)}}\ \partial_x \mn_\e(\alpha_1(h_\e)) \right) \label{b2aa}\\
& & + \partial_x \left( \left( D\ \frac{\beta_1'(\Gamma_\e)}{\beta_1'(\mn_\e(\Gamma_\e))} - \alpha_0(h_\e,\mn_\e(h_\e))\ \Gamma_\e\ \sigma'(\Gamma_\e) \right)\ \partial_x \Gamma_\e \right) \;\;\mbox{ in }\;\; Q_\infty\,, \nonumber \\
\partial_x h_\e & = & \partial_x \Gamma_\e = 0\;\;\mbox{ on }\;\; (0,\infty)\times \{0,L\}\,, \label{b3aa} \\
(h_\e,\Gamma_\e)(0) & = & (h_{0,\e},\Gamma_{0,\e}) \;\;\mbox{ in }\;\; (0,L)\,, \label{b4aa}
\end{eqnarray}
where $\Sigma_\e:=\Sigma_\e(h_\e,\Gamma_\e)$ solves
\begin{equation}
\Sigma_\e - \e^2\ \partial_x \left(\mn_\e(h_\e)\ \partial_x \Sigma_\e \right) = \sigma(\Gamma_\e) \;\;\mbox{ in }\;\; Q_\infty\,, \quad \partial_x \Sigma_\e = 0\;\;\mbox{ on }\;\; (0,\infty)\times \{0,L\}\, . \label{b10}
\end{equation}
This problem admits a unique global strong solution:

%%%%%%%%%%%%%%%%%%%%%%%%%%%%%%%%%%%%%%%%%%%%%%%%%%%%%%%%%%%%%%%%%%%%
\begin{theorem}\label{T1}
Let $h_0 , \Gamma_0\in W_2^2(0,L)$ be nonnegative functions satisfying \mbox{$\partial_x h_0(x)=\partial_x\Gamma_0(x)=0$} at $x=0,L$. For $\e\in (0,1)$ set
\begin{equation}
h_{0,\e} := h_0+\sqrt{\e}\,, \quad \Gamma_{0,\e} := \Gamma_0+\e\, . \label{b13}
\end{equation} 
Then there is a unique global nonnegative solution $(h_\e,\Gamma_\e)$ with
$$
h_\e, \Gamma_\e\in\mc^1\big((0,\infty),L_2(0,L)\big)\cap \mc\big((0,\infty),W_2^2(0,L)\big)\cap \mc \big([0,\infty),W_2^1(0,L)\big)
$$
to the regularized problem \eqref{b1aa}-\eqref{b4aa}. Moreover,
$$
h_\e(t,x)\ge \sqrt{\e}\ ,\quad \Gamma_\e(t,x)\ge \e\ ,\qquad (t,x)\in[0,\infty)\times (0,L)\ .
$$
\end{theorem}
%%%%%%%%%%%%%%%%%%%%%%%%%%%%%%%%%%%%%%%%%%%%%%%%%%%%%%%%%%%%%%%%%%%%

The remainder of this section is dedicated to the proof of this theorem.

%%%%%%%%%%%%%%%%%%%%%%%%%%%%%%%%%%%%%%%%%%%%%%%%%%%%%%%%%%%%%%%%%%%%
%%%%%%%%%%%%%%%%%%%%%%%%%%%%%%%%%%%%%%%%%%%%%%%%%%%%%%%%%%%%%%%%%%%%
\subsection{Local well-posedness}\label{sec:lwp}
%%%%%%%%%%%%%%%%%%%%%%%%%%%%%%%%%%%%%%%%%%%%%%%%%%%%%%%%%%%%%%%%%%%%
%%%%%%%%%%%%%%%%%%%%%%%%%%%%%%%%%%%%%%%%%%%%%%%%%%%%%%%%%%%%%%%%%%%%

We first focus our attention on the local solvability of the regularized problem. Given $\e\in (0,1)$ fixed we use the notation
$$
 V_{\mb}^\gamma:=W_{2,\mb}^\gamma(0,L)\cap \mc([0,L],D_0)
$$
with $D_0:=(\e^2,\infty)$ and $\gamma>1/2$, where $W_{2,\mb}^\gamma:=W_{2,\mb}^\gamma(0,L)$ coincides with the fractional Sobolev space $W_{2}^\gamma:=W_{2}^\gamma(0,L)$ if $\gamma\le 3/2$ or is the linear subspace thereof consisting of those $u\in W_{2}^\gamma$ satisfying the Neumann boundary conditions $\partial_x u(0)=\partial_x u(L)=0$ if $\gamma>3/2$. Observe that $V_{\mb}^\gamma$ is open in $W_{2,\mb}^\gamma$ and that $h_{0,\e}, \Gamma_{0,\e}\in V_\mb^2$. In the following we use the notation $\mc^{1-}$ to indicate that a function is locally Lipschitz continuous.

The proof of the next result about Nemitskii operators can be found, e.g., in \cite[Sect.15]{AmannAnnali88}:

%%%%%%%%%%%%%%%%%%%%%%%%%%%%%%%%%%%%%%%%%%%%%%%%%%%%%%%%%%%%%%%%%%%%
\begin{lemma}\label{L1}
Given $g\in \mc^2(D_0)$, let $g^\#(u)(x):=g(u(x))$, $x\in (0,L)$, for $u:(0,L)\rightarrow D_0$. Then $g^\#\in \mc^{1-}(V_\mb^\gamma,W_2^\gamma)$ for $\gamma\in (1/2,1)$.
\end{lemma}
%%%%%%%%%%%%%%%%%%%%%%%%%%%%%%%%%%%%%%%%%%%%%%%%%%%%%%%%%%%%%%%%%%%%

We shall also use the following continuity result about pointwise multiplication of real-valued functions.

%%%%%%%%%%%%%%%%%%%%%%%%%%%%%%%%%%%%%%%%%%%%%%%%%%%%%%%%%%%%%%%%%%%%
\begin{lemma}\label{L2}
(i) If $\gamma\ge 0$, then pointwise multiplication  $\mc^\gamma([0,L])\times \mc^\gamma([0,L])\rightarrow \mc^\gamma([0,L])$ is continuous.

(ii) If $\gamma>1/2$, then pointwise multiplication  $W_2^\gamma\times W_2^\gamma\rightarrow W_2^\gamma$ is continuous.

(iii) If $s>\gamma\ge 0$, then pointwise multiplication  $\mc^s([0,L])\times W_2^\gamma\rightarrow W_2^\gamma$ is continuous.
\end{lemma} 
%%%%%%%%%%%%%%%%%%%%%%%%%%%%%%%%%%%%%%%%%%%%%%%%%%%%%%%%%%%%%%%%%%%%

\begin{proof}
While assertion (i) is obvious, assertion (ii) is a consequence of \cite[Thm.4.1]{AmannMult},  and assertion (iii) is proved in \cite{Triebel} (see also \cite[Eq.(8.3)]{AmannTeubner93}).
\end{proof}

\medskip

The next proposition guarantees a nonnegative maximal solution to the regularized problem \eqref{b1aa}-\eqref{b4aa}.
The crucial point is that, though the local solution which we construct belongs to $W_{2,\mb}^2(0,L)$, an {\it a priori} estimate in $W_2^1$ is sufficient to obtain global existence, see \eqref{globex} below.

%%%%%%%%%%%%%%%%%%%%%%%%%%%%%%%%%%%%%%%%%%%%%%%%%%%%%%%%%%%%%%%%%%%%
\begin{proposition}\label{P3}
The regularized problem \eqref{b1aa}-\eqref{b4aa} admits a unique maximal strong solution $(h_\e,\Gamma_\e)$ on the maximal interval of existence $\mj:=\mj(\e)$. The solution possesses the regularity
$$
h_\e\, , \,\Gamma_\e \in\mc^1\big(\mj\setminus\{0\},L_2(0,L)\big)\cap \mc\big(\mj\setminus\{0\},W_{2,\mb}^2(0,L)\big)\cap\mc\big(\mj,W_2^1(0,L)\big)\ .
$$
Moreover, if for each $T>0$ there is some $c(T,\e)>0$ such that
\bqn\label{globex}
\min\big\{h_\e(t,x)\, , \, \Gamma_\e (t,x)\big\} \ge \e^2+c(T,\e)^{-1}\ ,\qquad \|h_\e(t)\|_{W_2^1}+\|\Gamma_\e(t)\|_{W_2^1}\le c(T,\e)
\eqn
for $t\in \mj\cap [0,T]$ and $x\in (0,L)$, then $\mj=[0,\infty)$, i.e. the solution exists globally.
\end{proposition}
%%%%%%%%%%%%%%%%%%%%%%%%%%%%%%%%%%%%%%%%%%%%%%%%%%%%%%%%%%%%%%%%%%%%

\begin{proof}
To establish the result we shall use the theory for quasilinear equations from \cite[Sect.13]{AmannTeubner93}. We simplify the notation by omitting the subscript $\e$ everywhere in \eqref{b1aa}-\eqref{b4aa} for the remainder of the proof. In the following we write $u:=(h,\Gamma)$ and introduce
$$
\mathpzc{a}(u):=\left(\begin{matrix} a_1(h) &0 \\ 0& \displaystyle{D\frac{\beta_1'(\Gamma)}{\beta_1'(\mn(\Gamma))} -\alpha_0(h,\mn(h))\Gamma \sigma'(\Gamma)}\end{matrix}\right)
$$
and
$$
F(u):=\partial_x\left(\mathpzc{b}(u)\partial_x\left(\begin{matrix} \Sigma (h,\Gamma)\\ \mn(\alpha_1(h)) \end{matrix}\right)\right)
$$
with 
$$
\mathpzc{b}(u):=\left(\begin{matrix} \displaystyle{-\frac{a_2(h) \sqrt{\mn(h)}}{\sqrt{h}}} &0 \\ 0& \displaystyle{G\frac{a_2(h) b_2(\Gamma) \sqrt{\alpha_0(h,\mn(h))}}{\sqrt{h a_1(h)}}} \end{matrix}\right)\ .
$$
Thus, setting
$$
\mathcal{A}(u)w:=-\partial_x(\mathpzc{a}(u)\partial_x w)\ ,\quad \mb w:=\partial_x w\ ,
$$
we may re-write equations \eqref{b1aa}-\eqref{b4aa} as a quasilinear problem of the form
\begin{align}
\partial_t u +\mathcal{A}(u) u& =  F(u) & \text{in}\ (0,\infty)\times (0,L)\ ,\label{c1aa}\\ 
\mb u& =  0& \text{on}\ (0,\infty)\times \{0,L\}\ ,\label{c2aa}\\
u(0,\cdot)& =  u^0:=(h_0,\Gamma_0) & \text{on}\ (0,L)\ .\label{c3aa} 
\end{align}
We next verify the assumptions of \cite[Thm.13.1]{AmannTeubner93} which then guarantees the existence of a weak solution to this quasilinear problem. Subsequently, we shall improve the regularity of the weak solution. In the following, we let $\xi\in (0,1/8)$ denote a sufficiently small number so that in particular, for $\gamma:=1/2-2\xi>0$,
\bqn\label{A0}
V_\mb^{1-\xi}\hookrightarrow \mc^\gamma :=\mc^\gamma([0,L])\ .
\eqn
Consequently, classical elliptic regularity applied to \eqref{b0} ensures
\bqn\label{A2}
\mn\in \mc^{1-}(V_\mb^{1-\xi},\mc^{\gamma+2})\qquad\text{with}\qquad \mn(f)\ge \e^2 \;\;\mbox{ for all }\;\;  f\in V_\mb^{1-\xi}\ .
\eqn
Lemma~\ref{L2}(i) and \eqref{b5}, \eqref{A0} easily yield
\bqn\label{A1}
\big[h\mapsto a_1(h)\big]\in \mc^{1-}(V_\mb^{1-\xi},\mc^\gamma)\ ,
\eqn
and we obtain from \eqref{a6}, \eqref{b7}, \eqref{A0}, \eqref{A2}, and Lemma~\ref{L2}(i) that
\bqn\label{A3}
\big[(h,\Gamma)\mapsto \alpha_0(h,\mn (h))\Gamma \sigma'(\Gamma)\big]\in \mc^{1-}(V_\mb^{1-\xi}\times V_\mb^{1-\xi},\mc^\gamma)\ ,
\eqn
while \eqref{b8}, \eqref{A0}, and \eqref{A2} entail
\bqn\label{A4}
\left[\Gamma\mapsto \frac{\beta_1'(\Gamma)}{\beta_1'(\mn(\Gamma))}\right]\in \mc^{1-}(V_\mb^{1-\xi},\mc^\gamma)\ .
\eqn
Thus, \eqref{A1}, \eqref{A3}, and \eqref{A4} imply 
\bqn\label{A5}
\big[u=(h,\Gamma)\mapsto \mathpzc{a}(u)\big]\in \mc^{1-}\big(V_\mb^{1-\xi}\times V_\mb^{1-\xi},(\mc^\gamma)^4\big)\ .
\eqn
Note that if $u=(h,\Gamma)$ with $h(x),\Gamma(x)>\e^2$ for $x\in (0,L)$, then the matrix $\mathpzc{a}(u(x))$ has strictly positive eigenvalues due to \eqref{a6}. Therefore, letting $2\hat{\alpha}:=3/2-3\xi$ so that $\gamma>2\hat{\alpha}-1$, and using the notion of \cite[Sect.4 \& 8]{AmannTeubner93} (in particular, see \cite[Eq.(8.6)]{AmannTeubner93}), it follows from \cite[Ex.4.3.e)]{AmannTeubner93} that
\bqn\label{A7}
(\mathcal{A},\mb)\in \mc^{1-}\big(V_\mb^{1-\xi}\times V_\mb^{1-\xi},\mathcal{E}^{\hat{\alpha}}((0,L))\big)\ .
\eqn
That is, $(\mathcal{A}(u),\mb)$ depends Lipschitz continuously on its argument $u\in V_\mb^{1-\xi}\times V_\mb^{1-\xi}$ and for each such $u$ fixed it is normally elliptic with operator $\mathcal{A}(u)$ in divergence form having $\mc^\gamma$-coefficients with $\gamma>2\hat{\alpha}-1$. We next study the regularity properties of the function $F$. Clearly, the function $g$, given by $g(r):=a_2(r)/\sqrt{r}$, $r>\e^2$, belongs to $\mc^2(D_0)$ by \eqref{b5} so that Lemma~\ref{L1} applies to yield
$$
\left[h\mapsto \frac{a_2(h)}{\sqrt{h}}\right]\in\mc^{1-}(V_\mb^{1-\xi},W_2^{1-\xi})\ .
$$
Since \eqref{A0} and \eqref{A2} provide
$$
\left[h\mapsto\sqrt{\mn(h)}\right]\in\mc^{1-}(V_\mb^{1-\xi},\mc^{\gamma+2})\ ,
$$
we obtain from Lemma~\ref{L2}(iii)
\bqn\label{A8}
\left[h\mapsto\frac{a_2(h)\sqrt{\mn(h)}}{\sqrt{h}}\right]\in \mc^{1-}\big(V_\mb^{1-\xi},W_2^{1-\xi}\big)\ .
\eqn 
As above we have by Lemma~\ref{L1}, \eqref{b5}, \eqref{b7}, and \eqref{A2}
\bqnn
\left[h\mapsto\frac{a_2(h)}{\sqrt{ha_1(h)}}\right]\in \mc^{1-}\big(V_\mb^{1-\xi},W_2^{1-\xi}\big)\ ,\qquad \left[h\mapsto\sqrt{\alpha_0(h,\mn(h))}\right]\in \mc^{1-}\big(V_\mb^{1-\xi},W_2^{1-\xi}\big)\ ,
\eqnn
from which we deduce, using \eqref{b6} and Lemma~\ref{L2}(ii),
\bqn\label{A9}
\left[u=(h,\Gamma)\mapsto G\frac{a_2(h) b_2(\Gamma) \sqrt{\alpha_0(h,\mn(h))}}{\sqrt{h a_1(h)}}\right]\in \mc^{1-}\big(V_\mb^{1-\xi}\times V_\mb^{1-\xi},W_2^{1-\xi}\big)\ .
\eqn 
Since $\alpha_1$ in \eqref{b7} is smooth in $D_0$, we get from \eqref{A0} and \eqref{A2}
\bqn\label{A10}
\left[h\mapsto\partial_x\mn(\alpha_1(h))\right]\in \mc^{1-}\big(V_\mb^{1-\xi},\mc^{1+\gamma}\big)\ .
\eqn
The operator $f\mapsto f-\e^2\partial_x(\mn(h)\partial_x f)$ is invertible in $\ml(\mc_\mb^{\gamma+2},\mc^\gamma)$ for $h\in V_\mb^{1-\xi}$ by \eqref{A2} and ellipticity, and it thus follows from \eqref{b10}, \eqref{A2}, the Lipschitz continuity (in fact: analyticity) of the inversion map $\ell\mapsto \ell^{-1}$ for linear operators, and $[\Gamma\mapsto\sigma(\Gamma)]\in\mc^{1-}(V_\mb^{1-\xi},\mc^\gamma)$ that
$$
\big[(h,\Gamma)\mapsto \partial_x\Sigma (h,\Gamma)\big]\in \mc^{1-}(V_\mb^{1-\xi}\times V_\mb^{1-\xi},\mc^{1+\gamma})\ .
$$
Combining this with \eqref{A8}, \eqref{A9}, and \eqref{A10} we derive from Lemma~\ref{L2}(iii) 
\bqn\label{A11}
F\in \mc^{1-}\big(V_\mb^{1-\xi}\times V_\mb^{1-\xi},W_2^{-\xi}\times W_2^{-\xi}\big)\ .
\eqn 
At this point observe that $W_2^{-\xi}=W_{2,\mb}^{-\xi}$ in the notation of \cite{AmannTeubner93} (in particular, see \cite[Eq.(7.5)]{AmannTeubner93}) since $\xi<1/2$. Thus, recalling that $2\hat{\alpha}=3/2-3\xi$ and choosing the numbers $(\tau,r,s,\sigma)$ in \cite[Eq.(13.2)]{AmannTeubner93} to be $(-\xi,1-\xi,1+\xi,2\hat{\alpha})$ we may apply \cite[Thm.13.1]{AmannTeubner93} due to \eqref{A7} and \eqref{A11}. We conclude that the quasilinear problem \eqref{c1aa}-\eqref{c3aa} with $u^0=(h_0,\Gamma_0)\in V_\mb^2\times V_\mb^2$ admits a unique maximal weak $W_2^{3/2- 3\xi}$-solution $(h,\Gamma)$ in the sense of \cite[Sect.13]{AmannTeubner93} on some interval $\mj$; that is,
$$
u=(h,\Gamma)\in \mc\big(\mj\setminus\{0\},W_{2,\mb}^{3/2-3\xi}\times W_{2,\mb}^{3/2-3\xi}\big) \cap \mc^1\big(\mj\setminus\{0\},W_{2,\mb}^{-1/2-3\xi}\times W_{2,\mb}^{-1/2-3\xi}\big)\ .
$$
The solution $u=(h,\Gamma)$ exists globally, i.e. $\mj=[0,\infty)$, provided that $(h,\Gamma)\vert_{[0,T]}$ is bounded in $W_2^1\times W_2^1$ and bounded away from the boundary of $V_\mb^{1-\xi}$ for each $T>0$. In particular, the solution exists globally provided \eqref{globex} holds.

We now aim at improving the regularity of $u=(h,\Gamma)$ as in \cite[Sect.14]{AmannTeubner93}. Given $\delta>0$ and $\xi>0$ still sufficiently small, set $\mj_\delta:=\mj\cap [\delta,\infty)$. Then
$$
h\, ,\, \Gamma\in \mc\big( \mj_\delta,W_{2,\mb}^{3/2-3\xi}\big) \cap \mc^1\big( \mj_\delta,W_{2,\mb}^{-1/2-3\xi}\big)\ ,
$$
from which we derive 
\bqn\label{A20}
h\, ,\, \Gamma \in \mc^{\rho}( \mj_\delta, W_{2,\mb}^{3/2-3\xi-2\rho})\ ,\quad 0\le 2\rho\le 2\ ,
\eqn
by \cite[Thm.7.2]{AmannTeubner93}. Taking $\rho:=\xi$ and setting $\mu:=1-6\xi$, we have
$W_{2,\mb}^{3/2-3\xi-2\rho}\hookrightarrow \mc^\mu$, and it thus follows from \eqref{A20} analogously to \eqref{A5} that
$$
\big[t\mapsto \mathpzc{a}(u(t))\big]\in \mc^\rho( \mj_\delta,(\mc^\mu)^4)\ .
$$
Hence, if we put $2\hat{\mu}:=2-8\xi$ so that $\mu>2\hat{\mu}-1$, we obtain similarly to \eqref{A7} from \cite[Ex.4.3.e), Eq.(8.6)]{AmannTeubner93} that
\bqn\label{A21}
(\mathcal{A}(u),\mb)\in\mc^\rho \big( \mj_\delta,\mathcal{E}^{\hat{\mu}}((0,L))\big)\ .
\eqn
Set then $2\nu:=3/2+\xi$ and note that, for $\xi>0$ small enough,
\bqn\label{A22}
3/2 < 2\nu < 2\hat{\mu} < 2\qquad\text{and}\qquad 2\rho=2\xi> \xi=2\nu-3/2\ .
\eqn
Also observe that \eqref{A11} and \cite[Eq.(7.5)]{AmannTeubner93} ensure
\bqn\label{A23}
F(u)\in\mc^\rho ( \mj_\delta, W_{2,\mb}^{-\xi}\times W_{2,\mb}^{-\xi})\hookrightarrow \mc^\rho ( \mj_\delta, W_{2,\mb}^{2\nu-2}\times W_{2,\mb}^{2\nu-2})\ .
\eqn
Gathering \eqref{A21}-\eqref{A23} and invoking \cite[Thm.11.3]{AmannTeubner93}, we conclude that the linear problem
\begin{align}
\partial_t v +\mathcal{A}(u(t)) v& =  F(u(t)) & \text{in}\ ( \mj_\delta\setminus\{\delta\})\times (0,L)\ ,\label{c1aaa}\\ 
\mb v& =  0& \text{on}\ ( \mj_\delta\setminus\{\delta\})\times \{0,L\}\ ,\label{c2aaa}\\
v(0,\cdot)& =  u( \delta,\cdot) & \text{on}\ (0,L)\ ,\label{c3aaa} 
\end{align}
has a unique strong $W_2^{2\nu}$-solution (in the sense of \cite[Sect.11]{AmannTeubner93})
$$
v\in\mc( \mj_\delta\setminus\{\delta\},W_{2,\mb}^{2\nu}\times W_{2,\mb}^{2\nu})\cap \mc^1( \mj_\delta\setminus\{\delta\},W_{2,\mb}^{2\nu-2} \times W_{2,\mb}^{2\nu-2})\ .
$$
Hence, $v$ and $u$ are both weak $W_2^{3/2-3\xi}$-solutions to \eqref{c1aaa}-\eqref{c3aaa} and thus $u=v$ by uniqueness of weak solutions to linear problems. Making $ \delta>0$ smaller we may replace $ \mj_\delta\setminus\{\delta\}$ by $ \mj_\delta$, and using the embedding $W_2^{3/2+\xi}\hookrightarrow \mc^{1+ \xi}$ for $\xi$ sufficiently small, we get
$
h,\Gamma\in\mc^\xi( \mj_\delta,\mc^{ 1+\xi})
$.
But then $u=(h,\Gamma)$ satisfies
$$
\partial_t u-\partial_x\big(\mathpzc{a}(u)\partial_x u\big)=F(u) \quad\text{on}\quad  \mj_\delta\times (0,L)
$$
subject to the boundary condition $\mb u=0$ with $\mathpzc{a}(u)\in \mc^{ \xi}( \mj_\delta,\mc^1)$ and $F(u)\in \mc^\xi( \mj_\delta, L_2)$ from which we readily conclude that
$$
h\, ,\, \Gamma\in\mc( \mj_\delta, W_{2,\mb}^2)\cap \mc^1( \mj_\delta, L_2)
$$
with $ \delta>0$ arbitrarily small by invoking \cite[Thm.10.1]{AmannTeubner93} with $(E_0,E_1):=(L_2,W_{2,\mb}^2)$. This proves the proposition.
\end{proof}

%%%%%%%%%%%%%%%%%%%%%%%%%%%%%%%%%%%%%%%%%%%%%%%%%%%%%%%%%%%%%%%%%%%%
%%%%%%%%%%%%%%%%%%%%%%%%%%%%%%%%%%%%%%%%%%%%%%%%%%%%%%%%%%%%%%%%%%%%
\subsection{Global well-posedness}\label{sec:gwp}
%%%%%%%%%%%%%%%%%%%%%%%%%%%%%%%%%%%%%%%%%%%%%%%%%%%%%%%%%%%%%%%%%%%%
%%%%%%%%%%%%%%%%%%%%%%%%%%%%%%%%%%%%%%%%%%%%%%%%%%%%%%%%%%%%%%%%%%%%

Let $(h_\e,\Gamma_\e)$ denote the unique strong solution to \eqref{b1aa}-\eqref{b4aa} on the maximal interval of existence $\mj=\mj(\e)$ provided by Proposition~\ref{P3}. We now show that \eqref{globex} holds which implies $\mj=[0,\infty)$. Introducing the abbreviations
\begin{equation}
H_\e := \mathcal{N}_\e(h_\e)\,, \quad A_\e := \mathcal{N}_\e(\alpha_1(h_\e))\,, \quad B_\e := \mathcal{N}_\e(\Gamma_\e)\,, \quad \Sigma_\e:=\Sigma_\e(h_\e,\Gamma_\e)\,, \label{b9} 
\end{equation}
and subsequently omitting the subscript $\e$ everywhere in \eqref{b1aa}-\eqref{b4aa} to simplify notation, the strong solution $(h,\Gamma)=(h_\e,\Gamma_\e)$ thus satisfies
\begin{eqnarray}
\partial_t h & = & \partial_x \left( a_1(h)\ \partial_x h - \frac{a_{2}(h) \sqrt{H}}{\sqrt{h}}\ \partial_x \Sigma \right) \;\;\mbox{ in }\;\;  \mj\setminus\{0\} \times (0,L)\,, \label{b1} \\
\partial_t \Gamma & = & \partial_x \left( G\ \frac{a_{2}(h) b_{2}(\Gamma) \sqrt{\alpha_0(h,H)}}{\sqrt{h a_1(h)}}\ \partial_x A \right) \label{b2}\\
& & + \partial_x \left( \left( D\ \frac{\beta_1'(\Gamma)}{\beta_1'(B)} - \alpha_0(h,H)\ \Gamma\ \sigma'(\Gamma) \right)\ \partial_x \Gamma \right) \;\;\mbox{ in }\;\;  \mj\setminus\{0\} \times (0,L)\,, \nonumber \\
\partial_x h & = & \partial_x \Gamma = 0\;\;\mbox{ on }\;\;  \mj\setminus\{0\} \times \{0,L\}\,, \label{b3} \\
(h,\Gamma)(0) & = & (h_{0},\Gamma_{0}) \;\;\mbox{ in }\;\; (0,L)\,. \label{b4}
\end{eqnarray}

We begin with some obvious consequences of the structure of \eqref{b1}-\eqref{b4}. 

%%%%%%%%%%%%%%%%%%%%%%%%%%%%%%%%%%%%%%%%%%%%%%%%%%%%%%%%%%%%%%%%%%%%
\begin{lemma}\label{leb101}
For $(t,x)\in\mathcal{J}\times (0,L)$, we have
\begin{eqnarray}
& & h(t,x) \ge \sqrt{\e}\,, \quad \Gamma(t,x) \ge \e\,, \label{b101} \\
& & \|h(t)\|_1 = \|h_{0}\|_1\,, \quad \|\Gamma(t)\|_1=\|\Gamma_{0}\|_1\,. \label{b102}
\end{eqnarray}
\end{lemma}
%%%%%%%%%%%%%%%%%%%%%%%%%%%%%%%%%%%%%%%%%%%%%%%%%%%%%%%%%%%%%%%%%%%%

\begin{proof}
Since $a_2(\sqrt{\e})=b_2(\e)=0$ by \eqref{b5} and \eqref{b6} and since $h_0\ge\sqrt{\e}$ and $\Gamma_0\ge\e$ by \eqref{b13}, the assertion \eqref{b101} is a straightforward consequence of the comparison principle applied separately to \eqref{b1} and \eqref{b2}. We next integrate \eqref{b1} and \eqref{b2} over $(0,t)\times (0,L)$ and use \eqref{b3} to obtain \eqref{b102}.
\end{proof}

In the next lemma, we collect several properties of $H$, $\Sigma$, $A$, and $B$.

%%%%%%%%%%%%%%%%%%%%%%%%%%%%%%%%%%%%%%%%%%%%%%%%%%%%%%%%%%%%%%%%%%%%
\begin{lemma}\label{leb102}
We have, for $(t,x)\in \mj\times (0,L)$,
\begin{eqnarray}
& & \|H(t)\|_p \le \|h(t)\|_p\,, \quad p\in [1,\infty]\,, \quad \sqrt{\e} \le H(t,x)\,, \label{b103a}\\
& & \|\partial_x H(t)\|_2^2 + 2\ \e^2\ \|\partial_x^2 H(t)\|_2^2 \le \|\partial_x h(t)\|_2^2\,, \label{b103}\\
& & \e^2\ \|\partial_x H(t)\|_\infty \le \|h(t)\|_1\,, \label{b103b}\\
& & \|A(t)\|_p \le \|\alpha_1(h(t))\|_p\,, \quad p\in [1,\infty]\,, \label{b104a}\\
& & \|\partial_x A(t)\|_2^2 + 2\ \e^2\ \|\partial_x^2 A(t)\|_2^2 \le \|\partial_x \alpha_1(h(t))\|_2^2\,, \label{b104}\\
& & \e^2\ \|\partial_x A(t)\|_\infty \le \|\alpha_1(h(t))\|_1\,, \label{b105}\\
& & \|B(t)\|_p \le \|\Gamma(t)\|_p\,, \quad p\in [1,\infty]\,, \quad \e \le B(t,x)\,, \label{b106}\\
& & \|\partial_x B(t)\|_2^2 + 2\ \e^2\ \|\partial_x^2 B(t)\|_2^2 \le \|\partial_x \Gamma(t)\|_2^2\,, \label{b107}\\
& & \e^2\ \|\partial_x B(t)\|_\infty \le \|\Gamma(t)\|_1\,, \label{b108}\\
& & \|\sqrt{H(t)} \partial_x \Sigma(t)\|_2^2 + 2\ \e^2\ \|\partial_x \left( H(t) \partial_x \Sigma(t) \right)\|_2^2 \le \|\sqrt{H(t)} \partial_x \sigma(\Gamma(t))\|_2^2\,, \label{b109}\\
& & \e^2\ \|H(t) \partial_x \Sigma(t)\|_\infty \le 2\|\sigma(\Gamma(t))\|_1\,. \label{b110}
\end{eqnarray}
\end{lemma}
%%%%%%%%%%%%%%%%%%%%%%%%%%%%%%%%%%%%%%%%%%%%%%%%%%%%%%%%%%%%%%%%%%%%

\begin{proof}
The first assertion of \eqref{b103a} follows from the classical contraction properties of $\mathcal{N}$ while the second is a consequence of \eqref{b9}, \eqref{b101}, and the comparison principle. We next deduce from \eqref{b9} that 
$$
\|\partial_x H(t)\|_2^2 + \e^2\ \|\partial_x^2 H(t)\|_2^2 = \int_0^L \partial_x h(t)\ \partial_x H(t)\ dx \le \frac{\|\partial_x H(t)\|_2^2 + \|\partial_x h(t)\|_2^2}{2}\,,
$$
from which \eqref{b103} follows. Finally, we infer from \eqref{b9} and the positivity of $H$ that $-\e^2\ \partial_x^2 H \le h$. For $(t,x)\in \mathcal{J}\times (0,L)$, we integrate the previous inequality first over $(0,x)$ and then over $(x,L)$, and use the homogeneous Neumann boundary conditions to obtain
$$
-\e^2\ \partial_x H(t,x) \le \int_0^x h(t,y)\ dy \le \|h(t)\|_1 \;\;\mbox{ and }\;\; \e^2\ \partial_x H(t,x) \le \int_x^L h(t,y)\ dy \le \|h(t)\|_1\,.
$$
Combining these two inequalities gives \eqref{b103b}. 

Next, the proofs of \eqref{b104a}-\eqref{b105} and \eqref{b106}-\eqref{b108} are similar to those of \eqref{b103a}-\eqref{b103b}. 

We now turn to $\Sigma$ and first notice that, since $H\ge \sqrt{\e}$ by \eqref{b103a}, the solution $\Sigma$ to \eqref{b10} belongs to $W_2^2(0,L)$. We thus may multiply \eqref{b10} by $(-\partial_x (H \partial_x\Sigma))$ and argue as in the proof of \eqref{b103} to establish \eqref{b109}. Finally, consider $(t,x)\in \mathcal{J}\times (0,L)$. As in the proof of \eqref{b103b}, we integrate \eqref{b10} first over $(0,x)$ and then over $(x,L)$, and use the homogeneous Neumann boundary conditions to obtain 
\begin{eqnarray*}
-\e^2\ H(t,x) \partial_x \Sigma(t,x) & \le & \int_0^x (\sigma(\Gamma)-\Sigma)(t,y)\ dy \le \|\sigma(\Gamma(t))\|_1+\|\Sigma(t)\|_1\,, \\
\e^2\ H(t,x) \partial_x \Sigma(t,x) & \le & \int_x^L (\sigma(\Gamma)-\Sigma)(t,y)\ dy \le \|\sigma(\Gamma(t))\|_1+\|\Sigma(t)\|_1\,.
\end{eqnarray*}
As \eqref{b10} implies that $\|\Sigma(t)\|_1\le \|\sigma(\Gamma(t))\|_1$ by classical approximation and monotonicity arguments, we obtain \eqref{b110}.
\end{proof}

We next define
\begin{eqnarray}
J_f & := & - \partial_x \alpha_1(h) + \frac{a_2(h) \sqrt{H}}{\sqrt{h a_1(h)}}\ \partial_x \Sigma\,, \label{b111} \\
J_s & := & \sqrt{\alpha_0(h,H)}\ \partial_x \sigma(\Gamma) - G\ \frac{a_2(h)}{\sqrt{h a_1(h)}}\ \frac{b_2(\Gamma)}{\Gamma}\ \partial_x A\,, \label{b112}
\end{eqnarray}
and show the existence of a Liapunov functional for the regularized problem \eqref{b1}-\eqref{b4} inherited from the one of \eqref{a1}-\eqref{a4}. 

%%%%%%%%%%%%%%%%%%%%%%%%%%%%%%%%%%%%%%%%%%%%%%%%%%%%%%%%%%%%%%%%%%%%
\begin{lemma}\label{leb103}
Given $t\in\mathcal{J}$, we have
\begin{equation}
\mathcal{L}(t) + \int_0^t \mathcal{D}(s)\ ds \le \mathcal{L}(0)\,, \label{b113} 
\end{equation}
with
\begin{equation}
\mathcal{L}(t) := \int_0^L \left[ \frac{G}{2}\ |h(t,x)|^2 + \phi(\Gamma(t,x)) \right]\ dx\,, \label{b114}
\end{equation}
\begin{equation}
\phi''(r) = - \frac{\sigma'(r)}{r} \ge 0\,, \quad \phi(1)=\phi'(1)=0\,, \label{b115}
\end{equation}
\begin{eqnarray}
2\mathcal{D}(t) & := &  G\ \|J_f(t)\|_2^2 + \|J_s(t)\|_2^2 + (\eta_1-\eta)\ \|\sqrt{H(t)} \partial_x\sigma(\Gamma(t))\|_2^2 \label{b116} \\
& & + (1-\eta_1)\ \|\sqrt{h(t)} \partial_x\sigma(\Gamma(t))\|_2^2 + (1-\eta) G\ \|\partial_x\alpha_1(h(t))\|_2^2\nonumber \\
& & + 2D\ \int_0^L \frac{\left| \partial_x\sigma(\Gamma(t)) \right|^2}{\beta_1'(B(t))}\ dx\,. \nonumber
\end{eqnarray}
\end{lemma}
%%%%%%%%%%%%%%%%%%%%%%%%%%%%%%%%%%%%%%%%%%%%%%%%%%%%%%%%%%%%%%%%%%%%

Observe that the last term in $\mathcal{D}(t)$ is well-defined as $\beta_1'(B)\ge \sigma_0 \e>0$ by \eqref{a6}, \eqref{b8}, and \eqref{b106}.

\begin{proof}
It follows from \eqref{b1}-\eqref{b3} that 
\begin{eqnarray*}
\frac{d\mathcal{L}}{dt} & = & G\ \int_0^L \partial_x h\ \left( - a_1(h)\ \partial_x h + \frac{a_2(h) \sqrt{H}}{\sqrt{h}}\ \partial_x \Sigma \right)\ dx \\
& & + \int_0^L \phi''(\Gamma)\ \partial_x\Gamma\ \left[ - D\ \frac{\beta_1'(\Gamma)}{\beta_1'(B)} + \alpha_0(h,H)\ \Gamma\ \sigma'(\Gamma) \right]\ \partial_x \Gamma\ dx \\
& & - G\ \int_0^L \phi''(\Gamma)\ \partial_x\Gamma\ \frac{a_2(h) b_2(\Gamma) \sqrt{\alpha_0(h,H)}}{\sqrt{h a_1(h)}}\ \partial_x A\ dx \\
& = & - \frac{G}{2}\ \|J_f\|_2^2 - \frac{1}{2}\ \|J_s\|_2^2 - D\ \int_0^L \frac{\left| \partial_x\sigma(\Gamma) \right|^2}{\beta_1'(B)}\ dx + \frac{1}{2}\ \mathcal{R}_f + \frac{G}{2}\ \mathcal{R}_s\,,
\end{eqnarray*}
with
\begin{eqnarray*}
\mathcal{R}_f & := & \int_0^L \left[ G\ \frac{a_2(h)^2 H}{h a_1(h)}\ \left| \partial_x \Sigma \right|^2  - \alpha_0(h,H)\ \left| \partial_x\sigma(\Gamma) \right|^2 \right]\ dx\,, \\
\mathcal{R}_s &:= &  \int_0^L \left[ G\ \frac{a_2(h)^2}{h a_1(h)}\ \frac{b_2(\Gamma)^2}{\Gamma^2}\ \left| \partial_x A\right|^2 - \left| \partial_x\alpha_1(h) \right|^2 \right]\ dx\,.
\end{eqnarray*}
On the one hand, it follows from \eqref{b14}, \eqref{b7}, and \eqref{b109} that
\begin{eqnarray*}
\mathcal{R}_f & \le & \int_0^L \left[ \eta\, H\, \left| \partial_x \Sigma \right|^2  - \alpha_0(h,H)\ \left| \partial_x\sigma(\Gamma) \right|^2 \right]\ dx \\
& \le & - (\eta_1-\eta)\ \left\| \sqrt{H} \partial_x\sigma(\Gamma) \right\|_2^2 - (1-\eta_1)\ \left\| \sqrt{h} \partial_x\sigma(\Gamma) \right\|_2^2\,.
\end{eqnarray*}
On the other hand, \eqref{b6}, \eqref{b14} and \eqref{b104} give
$$
\mathcal{R}_s \le \int_0^L \left[ \eta\ \left| \partial_x A\right|^2 - \left| \partial_x\alpha_1(h) \right|^2 \right]\ dx \le -(1-\eta)\ \left\| \partial_x\alpha_1(h) \right\|^2_{2}\,.
$$
Collecting the above inequalities yields Lemma~\ref{leb103} after integration with respect to time.
\end{proof}

We next estimate the $L_2$-norm of $\partial_x h$. While the previous estimates only depend mildly on $\e$, this will no longer be the case in the remainder of this section. In the following, the constants $C$, $C_j$, ... are independent of the free variables. Additional dependence on, say, $\e$ or $T>0$, we express explicitly by writing $C(\e)$, $C(\e,T)$, ...

%%%%%%%%%%%%%%%%%%%%%%%%%%%%%%%%%%%%%%%%%%%%%%%%%%%%%%%%%%%%%%%%%%%%
\begin{lemma}\label{leb104}
We define the function $\mathcal{A}_1$ by $\mathcal{A}_1'=a_1$ and $\mathcal{A}_1(0)=0$. For $T>0$ and $t\in\mathcal{J}\cap [0,T]$, we have
\begin{eqnarray}
\|h(t)\|_\infty + \left\| \partial_x\mathcal{A}_1(h(t)) \right\|_2 & \le & C_1(\e,T)\,, \label{b117} \\
\int_0^t \left(  \left\| \partial_t \alpha_1(h(s)) \right\|_2^2 + \|\partial_x h(s)\|_\infty^2 \right)\ ds & \le & C_1(\e,T)\,. \label{b118}
\end{eqnarray}
\end{lemma}
%%%%%%%%%%%%%%%%%%%%%%%%%%%%%%%%%%%%%%%%%%%%%%%%%%%%%%%%%%%%%%%%%%%%

\begin{proof}
Introducing $a_0(h):=h$ and 
$$
F_1 := - \left( \frac{a_2}{\sqrt{a_0}} \right)'(h)\ \partial_x h\ \sqrt{H}\ \partial_x\Sigma\,, \quad F_2 := \frac{a_2(h)}{2\sqrt{h}}\ \frac{\partial_x H}{\sqrt{H}}\ \partial_x\Sigma\,, \quad F_3 := - \frac{a_2(h)}{\sqrt{h H}}\ \partial_x\left( H \partial_x\Sigma \right)\,,
$$
equation \eqref{b1} reads
\begin{equation}
\partial_t h - \partial_x^2\mathcal{A}_1(h) = F_1 + F_2 +F_3\,. \label{b119}
\end{equation}
Recalling that $\alpha_1'=\sqrt{a_1}$, it follows from \eqref{b119} that
$$
\int_0^L \partial_t h\ \partial_t\mathcal{A}_1(h)\ dx + \int_0^L \partial_x\mathcal{A}_1(h)\ \partial_t\partial_x\mathcal{A}_1(h)\ dx = \int_0^L (F_1+F_2+F_3)\ \partial_t\mathcal{A}_1(h)\ dx\,,
$$
$$
\left\| \partial_t\alpha_1(h) \right\|_2^2 + \frac{1}{2}\ \frac{d}{dt} \left\| \partial_x\mathcal{A}_1(h) \right\|_2^2 \le \frac{1}{2}\ \left\| \partial_t\alpha_1(h) \right\|_2^2 + \frac{1}{2}\ \int_0^L a_1(h)\ (F_1+F_2+F_3)^2\ dx\ ,
$$
\begin{equation}
\left\| \partial_t\alpha_1(h) \right\|_2^2 + \frac{d}{dt} \left\| \partial_x\mathcal{A}_1(h) \right\|_2^2 \le 3\ \sum_{i=1}^3 \left\| \sqrt{a_1(h)}\ F_i \right\|_2^2\,. \label{b120}
\end{equation}
To estimate the term involving $F_1$, we write
\begin{eqnarray*}
\left\| \sqrt{a_1(h)}\ F_1 \right\|_2 & = & \left\| \left( \frac{a_2}{\sqrt{a_0}} \right)'(h)\ \frac{\partial_x \mathcal{A}_1(h)}{\sqrt{a_1(h)}}\ \sqrt{H}\ \partial_x\Sigma\right\|_2 \\
& \le & \left\| \left( \frac{a_2}{\sqrt{a_0}} \right)'(h)\ \frac{1}{\sqrt{a_1(h)}} \right\|_\infty\ \left\| \partial_x \mathcal{A}_1(h) \right\|_2\ \left\| \sqrt{H}\ \partial_x\Sigma\right\|_\infty\,,
\end{eqnarray*}
and observe that \eqref{a6b}, \eqref{b102}, \eqref{b103a}, and \eqref{b110} ensure that
$$
\left\| \sqrt{H}\ \partial_x\Sigma\right\|_\infty \le \frac{\left\| H\ \partial_x\Sigma\right\|_\infty}{\e^{1/4}} \le \frac{2\ \|\sigma(\Gamma)\|_1}{\e^{9/4}} \le C(\e)\ (1+\|\Gamma\|_1)\le C(\e)\,,
$$
while we infer from \eqref{b5} and \eqref{b101} that
$$
\left\| \left( \frac{a_2}{\sqrt{a_0}} \right)'(h)\ \frac{1}{\sqrt{a_1(h)}} \right\|_\infty = \sqrt{\frac{3}{G}}\ \left\| \frac{3 h^2 - 2 \sqrt{\e} h - \e}{4 h^3} \right\|_\infty \le \frac{C}{\sqrt{\e}}\,.
$$
Consequently,
\begin{equation}
\left\| \sqrt{a_1(h)}\ F_1 \right\|_2 \le C(\e)\ \left\| \partial_x \mathcal{A}_1(h) \right\|_2\,. \label{b121}
\end{equation}

We next turn to the term involving $F_2$ and deduce from \eqref{b14} and \eqref{b103a} that
\begin{eqnarray*}
\left\| \sqrt{a_1(h)}\ F_2 \right\|_2 & = & \left\| \frac{a_2(h)}{2\sqrt{h a_1(h)}}\ a_1(h)\ \frac{\partial_x H}{H^{3/2}}\ H \partial_x\Sigma \right\|_2 \\
& \le & \left\| \frac{a_2(h)}{2\sqrt{h a_1(h)}} \right\|_\infty\ \|a_1(h)\|_2\ \frac{\|\partial_x H\|_\infty}{\e^{3/4}}\ \left\| H \partial_x\Sigma \right\|_\infty \\
& \le & \left( \frac{\eta}{4G \e^{3/2}} \right)^{1/2}\ \|a_1(h)\|_2\ \|\partial_x H\|_\infty\ \left\| H \partial_x\Sigma \right\|_\infty\,.
\end{eqnarray*}
Owing to \eqref{a6b}, \eqref{b102}, \eqref{b103b}, and \eqref{b110}, we obtain
$$
\left\| \sqrt{a_1(h)}\ F_2 \right\|_2 \le C(\e)\ \|a_1(h)\|_2\ \frac{\| h\|_1}{\e^2}\ \frac{2\ \left\| \sigma(\Gamma) \right\|_1}{\e^2} \le C(\e)\ \|a_1(h)\|_2\,.
$$
Since
\begin{equation}
a_1(h) = \frac{4\mathcal{A}_1(h)}{h} \le \frac{4 \mathcal{A}_1(h)}{\sqrt{\e}} \label{b122}
\end{equation}
by \eqref{b5} and \eqref{b101}, we end up with
\begin{equation}
\left\| \sqrt{a_1(h)}\ F_2 \right\|_2 \le C(\e)\ \|\mathcal{A}_1(h)\|_\infty\,. \label{b123}
\end{equation}
Finally, by \eqref{b14}, \eqref{b103a}, \eqref{b109}, \eqref{b116}, and \eqref{b122}, we have 
\begin{eqnarray*}
\|\sqrt{a_1(h)}\ F_3\|_2 & = & \left\| \frac{a_2(h)}{\sqrt{h a_1(h)}}\ \frac{\sqrt{a_1(h)}}{\sqrt{H}}\ \partial_x\left( H \partial_x\Sigma \right) \right\|_2 \\
& \le & \left( \frac{\eta}{G \e^{1/2}} \right)^{1/2}\ \|a_1(h)\|_\infty^{1/2}\ \left\| \partial_x\left( H \partial_x\Sigma \right) \right\|_2 \\
& \le & C(\e)\ \|a_1(h)\|_\infty^{1/2}\ \frac{\left\| \sqrt{H} \partial_x\sigma(\Gamma) \right\|_2}{\e} \, ,
\end{eqnarray*}
whence
\begin{equation}
\|\sqrt{a_1(h)}\ F_3\|_2 \le C(\e)\ \mathcal{D}^{1/2}\ \|\mathcal{A}_1(h)\|_\infty^{1/2}\,. \label{b124}
\end{equation}
It then follows from \eqref{b120}, \eqref{b121}, \eqref{b123}, and \eqref{b124} that
\begin{equation}
\left\| \partial_t\alpha_1(h) \right\|_2^2 + \frac{d}{dt} \left\| \partial_x\mathcal{A}_1(h) \right\|_2^2 \le C(\e)\ \left( \left\|\partial_x\mathcal{A}_1(h) \right\|_2^2 + \|\mathcal{A}_1(h)\|_\infty^2 + \mathcal{D}\ \|\mathcal{A}_1(h)\|_\infty \right)\,. \label{b125}
\end{equation}
Owing to \eqref{b5} and $\mathcal{A}_1'=a_1$, we have, for $(t,x)\in\mathcal{J}\times (0,L)$,
\begin{eqnarray*}
0 \le L \mathcal{A}_1(h(t,x)) & \le & \|\mathcal{A}_1(h(t))\|_1 + L^{3/2}\ \left\| \partial_x \mathcal{A}_1(h(t)) \right\|_2 \\
& \le & \left( \frac{G}{12} \right)^{1/4}\ \int_0^L h(t,y)\ \mathcal{A}_1(h(t,y))^{3/4}\ dy + L^{3/2}\ \left\| \partial_x \mathcal{A}_1(h(t)) \right\|_2 \\
& \le & \left( \frac{G}{12} \right)^{1/4}\ \|h\|_1\ \|\mathcal{A}_1(h(t))\|_\infty^{3/4} + L^{3/2}\ \left\| \partial_x \mathcal{A}_1(h(t)) \right\|_2 \,,
\end{eqnarray*}
and we thus infer from \eqref{b102} and Young's inequality that 
$$
0 \le L \mathcal{A}_1(h(t,x)) \le \frac{3L}{4}\ \|\mathcal{A}_1(h(t)) \|_\infty + \frac{G \|h_0\|_1^4}{48 L^3} + L^{3/2}\ \left\| \partial_x \mathcal{A}_1(h(t)) \right\|_2\,,
$$
whence
\begin{equation}
\|\mathcal{A}_1(h(t)) \|_\infty \le C\ \left( 1 + \left\| \partial_x \mathcal{A}_1(h(t)) \right\|_2 \right)\,, \quad t\in\mathcal{J}\,.\label{b126}
\end{equation}
Inserting this inequality in \eqref{b125} gives
$$
\left\| \partial_t\alpha_1(h) \right\|_2^2 + \frac{d}{dt} \left\| \partial_x\mathcal{A}_1(h) \right\|_2^2 \le C(\e)\ \left( 1 + \mathcal{D} \right)\ \left( 1 + \left\|\partial_x\mathcal{A}_1(h) \right\|_2^2 \right)\,,
$$
from which we conclude that, for $t\in\mathcal{J}\cap [0,T]$, 
$$
\left\| \partial_x\mathcal{A}_1(h(t)) \right\|_2^2 + \int_0^t \left\| \partial_t\alpha_1(h(s)) \right\|_2^2\ ds \le \left( 1 + \left\| \partial_x\mathcal{A}_1(h_0) \right\|_2^2 \right)\ \exp{\left\{ C(\e)\ \left( t+ \int_0^t \mathcal{D}(s)\ ds \right) \right\}} \, ,
$$
hence by \eqref{b113}
\begin{equation}
\left\| \partial_x\mathcal{A}_1(h(t)) \right\|_2^2 + \int_0^t \left\| \partial_t\alpha_1(h(s)) \right\|_2^2\ ds \le C(\e,T)\ . \label{b127}
\end{equation}

A first consequence of \eqref{b126} and \eqref{b127} is that \eqref{b117} holds true. Next, recalling \eqref{b3}, \eqref{b101}, \eqref{b119}, \eqref{b121}, \eqref{b123}, and \eqref{b124}, we obtain for $t\in \mathcal{J}\cap [0,T]$ and $x\in (0,L)$:
\begin{eqnarray*}
\left| \partial_x \mathcal{A}_1(h(t,x)) \right| & = & \left| \int_0^x \partial_x^2 \mathcal{A}_1(h(t,y))\ dy \right| = \left| \int_0^x \left( \sum_{i=1}^3 F_i -\partial_t h \right)(t,y)\ dy  \right| \\
& \le & \int_0^L \frac{1}{\sqrt{a_1(h)}}\ \left[ |\partial_t \alpha_1(h)| + \sum_{i=1}^3 \sqrt{a_1(h)}\ |F_i| \right]\ dy \\
& \le & C(\e)\ \left( \left\| \partial_t \alpha_1(h) \right\|_2 + \sum_{i=1}^3 \left\| \sqrt{a_1(h)}\ F_i \right\|_2 \right) \\
& \le & C(\e)\ \left( \left\| \partial_t \alpha_1(h) \right\|_2 + \left\|\partial_x\mathcal{A}_1(h) \right\|_2 + \|\mathcal{A}_1(h)\|_\infty + \mathcal{D}^{1/2}\ \|\mathcal{A}_1(h)\|_\infty^{1/2} \right) \,,
\end{eqnarray*}
and we infer from \eqref{b126} and \eqref{b127} that 
$$
\left\| \partial_x \mathcal{A}_1(h(t)) \right\|_\infty  \le C(\e,T)\ \left( 1 + \left\| \partial_t \alpha_1(h(t)) \right\|_2 + \mathcal{D}(t)^{1/2} \right)\,.
$$
Using next \eqref{b113} and \eqref{b127} we obtain
$$
\int_0^t \left\| \partial_x \mathcal{A}_1(h(s)) \right\|_\infty^2\ ds \le C(\e,T)\ \left[ 1 + \int_0^t \left( \left\| \partial_t \alpha_1(h(s)) \right\|_2^2 + \mathcal{D}(s) \right)\ ds \right] \le C(\e,T)\,.
$$
Since 
$$
\left| \partial_x h \right| = \frac{\left| \partial_x \mathcal{A}_1(h) \right|}{a_1(h)} \le \frac{3}{G \e^{3/2}}\ \left\| \partial_x \mathcal{A}_1(h) \right\|_\infty 
$$
by \eqref{b5} and \eqref{b101}, the estimate \eqref{b118} follows from the above analysis and \eqref{b127}.
\end{proof}

We now improve the estimates on $\Gamma$ and begin with an $L_\infty$-bound.

\begin{lemma}\label{leb105}
Given $T>0$ and $t\in\mathcal{J}\cap [0,T]$, we have
\begin{equation}
\|\Gamma(t)\|_\infty \le C_2(\e,T)\,. \label{b128}
\end{equation}
\end{lemma}

\begin{proof}
Let $T>0$ and $t\in\mathcal{J}\cap [0,T]$ be given. Define again $a_0(h)=h$,
$$
q := - \frac{a_2(h)}{\sqrt{a_0(h) a_1(h)}}\ \sqrt{\alpha_0(h,H)}\ \partial_x A \;\;\mbox{ and }\;\; q_1 := \partial_x q\,,
$$
and the parabolic operator
\begin{eqnarray*}
\mathcal{P}w & := & \partial_t w - \partial_x \left[ \left( D\ \frac{\beta_1'(\Gamma)}{\beta_1'(B)}  - \alpha_0(h,H)\ \Gamma \sigma'(\Gamma) \right)\ \partial_x w \right] + G\ q\ \partial_x w + G\ q_1\ b_2(w)\,,
\end{eqnarray*}
so that \eqref{b2} also reads
\begin{equation}
\mathcal{P}\Gamma = 0  \;\;\mbox{ in }\;\; \mj\setminus\{0\} \times (0,L)\,. \label{b129}
\end{equation}
We next observe that $q_1 = q_{11} + q_{12} + q_{13}$ with
\begin{eqnarray*}
q_{11} & := & - \left( \frac{a_2}{\sqrt{a_0 a_1}} \right)'(h)\ \partial_x h\ \sqrt{\alpha_0(h,H)}\ \partial_x A\,, \\
q_{12} & := & - \frac{a_2(h)}{2\sqrt{a_0(h) a_1(h)}}\ \frac{\eta_1 \partial_x H + (1-\eta_1) \partial_x h}{\sqrt{\alpha_0(h,H)}}\ \partial_x A\,, \\
q_{13} & := & - \frac{a_2(h)}{\sqrt{a_0(h) a_1(h)}}\ \sqrt{\alpha_0(h,H)}\ \partial_x^2 A\,.
\end{eqnarray*}
By \eqref{b101}, \eqref{b103a}, \eqref{b105}, and \eqref{b117}, we have
\begin{eqnarray}
 \|q_{11}\|_\infty & \le & \left( \frac{3\e}{G} \right)^{1/2}\  \left\| \frac{h-\sqrt{\e}}{h^3} \right\|_\infty\ \left\| \partial_x h \right\|_\infty\ \left( \eta_1\ \|H\|_\infty + (1-\eta_1)\ \|h\|_\infty \right)^{1/2}\ \left\| \partial_x A \right\|_\infty \nonumber\\
& \le & C(\e)\ \left\| \partial_x h \right\|_\infty\ \|h\|_\infty^{1/2}\ \frac{\|\alpha_1(h)\|_1}{\e^2}\ ,\quad \nonumber 
\end{eqnarray}
that is,
\begin{equation}
 \|q_{11}\|_\infty  \le  C(\e,T)\ \left\| \partial_x h \right\|_\infty\,. \label{b130}
\end{equation}
We next infer from \eqref{b14}, \eqref{b101}, \eqref{b102}, \eqref{b103a}, \eqref{b103b}, \eqref{b105}, and \eqref{b117} that
\begin{eqnarray}
 \|q_{12}\|_\infty & \le & \left( \frac{\eta}{4G} \right)^{1/2}\ \frac{\eta_1 \left\| \partial_x H \right\|_\infty + (1-\eta_1) \left\| \partial_x h \right\|_\infty}{\e^{1/4}}\ \left\| \partial_x A \right\|_\infty \nonumber \\
& \le & C(\e)\ \left( \frac{\|h_0\|_1}{\e^2} + \left\|\partial_x h \right\|_\infty \right)\ \frac{\|\alpha_1(h)\|_1}{\e^2} \ ,\nonumber 
\end{eqnarray}
that is,
\begin{equation}
 \|q_{12}\|_\infty  \le  C(\e,T)\ \left( 1 + \left\| \partial_x h(t) \right\|_\infty \right)\,. \label{b132}
\end{equation}
It finally follows from \eqref{b0}, \eqref{b14}, \eqref{b9}, \eqref{b103a}, \eqref{b104a}, and \eqref{b117} that
$$
|q_{13}| \le \left( \frac{\eta}{G} \right)^{1/2}\ \|h\|_\infty^{1/2}\ \frac{|\alpha_1(h)-A|}{\e^2} \le C(\e,T)\ \left( \|\alpha_1(h)\|_\infty + \|A\|_\infty \right) \le C(\e,T)\ \|\alpha_1(h)\|_\infty \,,
$$ 
and thus
\begin{equation}
\|q_{13}\|_\infty \le C(\e,T)\,. \label{b133}
\end{equation}
Combining \eqref{b130}, \eqref{b132}, and \eqref{b133}, we conclude that 
$$
\|q_1(t)\|_\infty \le C(\e,T)\ \left( 1 + \left\| \partial_x h(t) \right\|_\infty \right)\,,
$$
which by \eqref{b118} gives

\begin{equation}
\int_0^t \|q_1(s)\|_\infty\ ds \le C(\e,T)\,, \quad t\in\mathcal{J}\cap [0,T]\,. \label{b134}
\end{equation}

Now, let $Q$ be the solution to the ordinary differential equation
\begin{equation}
\frac{dQ}{dt}(t) - G\ \|q_1(t)\|_\infty\ b_2(Q(t)) = 0\,, \quad t\in\mathcal{J}\,, \label{b135}
\end{equation}
with initial condition $Q(0):=\|\Gamma_0\|_\infty\ge \e$. We clearly have $Q(t)\ge \e$ for $t\in\mathcal{J}$ and
$$
\mathcal{P}Q(t) = G\ \left( \|q_1(t)\|_\infty + q_1(t ,x) \right)\ b_2(Q(t)) \ge 0\,, \quad  (t,x)\in\mathcal{J}\setminus\{0\}\times (0,L)\,.
$$
Recalling \eqref{b129}, the comparison principle entails that 
\begin{equation}
\Gamma(t,x) \le Q(t)\,, \quad (t,x)\in\mathcal{J}\times [0,L]\,. \label{b136}
\end{equation}
Since $b_2(Q)\le Q$, we deduce from \eqref{b134}, \eqref{b135}, and \eqref{b136} that, for $T>0$ and $t\in\mathcal{J}\cap [0,T]$,
$$
\|\Gamma(t)\|_\infty \le Q(t) \le Q(0)\ \exp{\left\{ G\ \int_0^t \|q_1(s)\|_\infty\ ds \right\}} \le C(\e,T)\,,
$$
as expected.
\end{proof}

The final step of the proof of Theorem~\ref{T1} is an $L_2$-estimate on $\partial_x \Gamma$.

\begin{lemma}\label{leb106}
For $T>0$ and $t\in\mathcal{J}\cap [0,T]$, we have
\begin{equation}
\left\| \partial_x\Gamma(t) \right\|_2 \le C_3(\e,T)\,. \label{b137}
\end{equation}
\end{lemma}

\begin{proof}
Introducing $a_0(h)=h$ and $\gamma:=\partial_x \beta_1(\Gamma)$, it follows from \eqref{b8} and \eqref{b2} that 
$$
\partial_t\beta_1(\Gamma) - \beta_1'(\Gamma)\ \partial_x \left( G\ \frac{a_2(h) b_2(\Gamma) \sqrt{\alpha_0(h,H)}}{\sqrt{a_0(h) a_1(h)}}\ \partial_x A + \left( \frac{D}{\beta_1'(B)} + \alpha_0(h,H) \right)\ \gamma \right) =0\,.
$$
Differentiating with respect to $x$ we obtain
$$
\partial_t\gamma - \partial_x \left[ \beta_1'(\Gamma)\ \partial_x \left( G\ \frac{a_2(h) b_2(\Gamma) \sqrt{\alpha_0(h,H)}}{\sqrt{a_0(h) a_1(h)}}\ \partial_x A + \left( \frac{D}{\beta_1'(B)} + \alpha_0(h,H) \right)\ \gamma \right) \right]=0\,.
$$
Since $\gamma(t,x)=0$ for $(t,x)\in\mathcal{J}\times\{0,L\}$ by \eqref{b3}, we deduce from the above equation that
\begin{equation}
\frac{1}{2}\ \frac{d}{dt} \|\gamma\|_2^2 + \int_0^L \left( \frac{D}{\beta_1'(B)} + \alpha_0(h,H) \right)\ \beta_1'(\Gamma)\ \left| \partial_x \gamma \right|^2\ dx = \sum_{i=1}^5 Y_i\,, \label{b138} 
\end{equation} 
with
\begin{eqnarray*}
Y_1 & := & - \int_0^L \beta_1'(\Gamma)\ \partial_x\gamma \left[ - D\ \frac{\beta_1''(B)}{\beta_1'(B)^2}\ \partial_x B + \eta_1\ \partial_x H + (1-\eta_1)\ \partial_x h \right]\ \gamma\ dx\,, \\
Y_2 & := & - G\ \int_0^L \beta_1'(\Gamma)\ \partial_x\gamma\ \left( \frac{a_2}{\sqrt{a_0 a_1}}\right)'(h)\ \partial_x h\ b_2(\Gamma)\ \sqrt{\alpha_0(h,H)}\ \partial_x A\ dx\,, \\
Y_3 &:= & - G\ \int_0^L \beta_1'(\Gamma)\ \partial_x\gamma\ \frac{a_2(h)}{\sqrt{a_0(h) a_1(h)}}\ \frac{b_2'(\Gamma)}{\beta_1'(\Gamma)}\ \gamma\ \sqrt{\alpha_0(h,H)}\ \partial_x A\ dx\,, \\
Y_4 & := & - G\ \int_0^L \beta_1'(\Gamma)\ \partial_x\gamma\ \frac{a_2(h)}{\sqrt{a_0(h) a_1(h)}}\ b_2(\Gamma)\ \frac{\eta_1 \partial_x H + (1 - \eta_1)\ \partial_x h}{2\sqrt{\alpha_0(h,H)}}\ \partial_x A\ dx\,, \\
Y_5 & := & - G\ \int_0^L \beta_1'(\Gamma)\ \partial_x\gamma\ \frac{a_2(h)}{\sqrt{a_0(h) a_1(h)}}\ b_2(\Gamma)\ \sqrt{\alpha_0(h,H)}\ \partial_x^2 A\ dx\,.
\end{eqnarray*}

We now estimate each of the terms $Y_i$, $1\le i \le 5$, separately for $T>0$ and $t\in\mathcal{J}\cap [0,T]$. By \eqref{b102}, \eqref{b103b}, \eqref{b106}, \eqref{b108}, and \eqref{b128}, we have
\begin{eqnarray*}
|Y_1| & \le & \left\| \sqrt{\beta_1'(\Gamma)}\ \partial_x\gamma \right\|_2\ \|\beta_1'(\Gamma)\|_\infty^{1/2} \left[ D\ \left\| \frac{\beta_1''(B)}{\beta_1'(B)^2} \right\|_\infty\ \left\| \partial_x B \right\|_\infty + \left\| \partial_x H \right\|_\infty + \left\| \partial_x h \right\|_\infty \right]\ \|\gamma\|_2 \nonumber\\
& \le & C(\e,T)\ \left[ \frac{\|\Gamma\|_1}{\e^2} + \frac{\|h\|_1}{\e^2} + \left\| \partial_x h \right\|_\infty \right]\ \|\gamma\|_2\ \left\| \sqrt{\beta_1'(\Gamma)} \ \partial_x\gamma \right\|_2\ , \nonumber
\end{eqnarray*}
that is,
\begin{equation}
|Y_1|  \le  \frac{\e}{5}\ \left\| \sqrt{\beta_1'(\Gamma)} \ \partial_x\gamma \right\|_2^2 + C(\e,T)\ \left( 1 + \left\| \partial_x h \right\|_\infty^2 \right)\ \|\gamma\|_2^2\,. \label{b139}
\end{equation}
It next follows from \eqref{b5}, \eqref{b101}, \eqref{b103a}, \eqref{b104}, \eqref{b117}, and \eqref{b128} that 
\begin{eqnarray*}
|Y_2| & \le & G\ \left\| \sqrt{\beta_1'(\Gamma)}\ \partial_x\gamma \right\|_2\ \|(\sqrt{\beta_1'} b_2)(\Gamma)\|_\infty\ \left\| \left( \frac{a_2}{\sqrt{a_0 a_1}}\right)'(h) \right\|_\infty\ \left\| \partial_x h \right\|_\infty\ \|\alpha_0(h,H)\|_\infty^{1/2}\ \left\| \partial_x A \right\|_2 \nonumber \\
& \le & C(\e,T)\ \left\| \sqrt{\beta_1'(\Gamma)}\ \partial_x\gamma \right\|_2\ \left( \frac{3\e}{G} \right)^{1/2}\ \left\| \frac{h-\sqrt{\e}}{h^3}\right\|_\infty\ \left\| \partial_x h \right\|_\infty\ \|h\|_\infty^{1/2}\ \left\| \partial_x \alpha_1(h) \right\|_2\nonumber \\
& \le & C(\e,T)\ \left\| \partial_x h \right\|_\infty\ \left\| \sqrt{\beta_1'(\Gamma)}\ \partial_x\gamma \right\|_2\ , 
\nonumber 
\end{eqnarray*}
so that
\begin{equation}
|Y_2|  \le  \frac{\e}{5}\ \left\| \sqrt{\beta_1'(\Gamma)} \ \partial_x\gamma \right\|_2^2 + C(\e,T)\ \left\| \partial_x h \right\|_\infty^2 \,, \label{b140}
\end{equation}
while \eqref{b14}, \eqref{b101}, \eqref{b102}, \eqref{b103a}, \eqref{b105}, and \eqref{b117} ensure that 
\begin{eqnarray}
|Y_3| & \le & \sqrt{\eta G}\ \left\| \sqrt{\beta_1'(\Gamma)}\ \partial_x\gamma \right\|_2\ \left\| \frac{1}{\sqrt{\beta_1'(\Gamma)}} \right\|_\infty\ \|\gamma\|_2\ \|\alpha_0(h,H)\|_\infty^{1/2}\ \left\| \partial_x A \right\|_\infty \nonumber \\
& \le & C(\e,T)\ \|h\|_\infty^{1/2}\ \frac{\| \alpha_1(h) \|_1}{\e^2}\ \|\gamma\|_2\ \left\| \sqrt{\beta_1'(\Gamma)}\ \partial_x\gamma \right\|_2\ ,  \nonumber 
\end{eqnarray}
whence
\begin{equation}
|Y_3|  \le  \frac{\e}{5}\ \left\| \sqrt{\beta_1'(\Gamma)} \ \partial_x\gamma \right\|_2^2 + C(\e,T)\ \|\gamma\|_2^2 \,. \label{b141}
\end{equation}
Finally, owing to \eqref{b14}, \eqref{b101}, \eqref{b102}, \eqref{b103a}, \eqref{b103b}, \eqref{b104}, \eqref{b117}, and \eqref{b128}, we have
\begin{eqnarray}
|Y_4| & \le & \sqrt{\eta G}\ \left\| \sqrt{\beta_1'(\Gamma)}\ \partial_x\gamma \right\|_2\ \|(\sqrt{\beta_1'} b_2)(\Gamma)\|_\infty\ \frac{\left\| \partial_x H \right\|_\infty + \left\| \partial_x h \right\|_\infty}{\e^{1/4}}\ \left\| \partial_x A \right\|_2 \nonumber \\
& \le & C(\e,T)\ \left( \frac{\|h\|_1}{\e^2} + \left\| \partial_x h \right\|_\infty \right)\ \left\| \partial_x \alpha_1(h) \right\|_2\ \left\| \sqrt{\beta_1'(\Gamma)}\ \partial_x\gamma \right\|_2\ , \nonumber 
\end{eqnarray}
that is,
\begin{equation}
|Y_4|  \le  \frac{\e}{5}\ \left\| \sqrt{\beta_1'(\Gamma)} \ \partial_x\gamma \right\|_2^2 + C(\e,T)\ \left( 1 + \left\| \partial_x h \right\|_\infty^2 \right) \,, \label{b142}
\end{equation}
and
\begin{eqnarray}
|Y_5| & \le & \sqrt{\eta G}\ \left\| \sqrt{\beta_1'(\Gamma)}\ \partial_x\gamma \right\|_2\ \|(\sqrt{\beta_1'} b_2)(\Gamma)\|_\infty\ \|\alpha_0(h,H)\|_\infty^{1/2}\ \left\| \partial_x^2 A \right\|_2 \nonumber \\
& \le & C(\e,T)\ \|h\|_\infty^{1/2}\ \frac{\left\| \partial_x \alpha_1(h) \right\|_2} \e\ \left\| \sqrt{\beta_1'(\Gamma)}\ \partial_x\gamma \right\|_2\ , 
\end{eqnarray}
that is,
\begin{equation}
|Y_5|  \le  \frac{\e}{5}\ \left\| \sqrt{\beta_1'(\Gamma)} \ \partial_x\gamma \right\|_2^2 + C(\e,T)\,. \label{b143}
\end{equation}
Collecting \eqref{b139}-\eqref{b143}, we deduce from \eqref{b101}, \eqref{b103a}, and \eqref{b138} that 
\begin{eqnarray*}
& & \frac{1}{2}\ \frac{d}{dt} \|\gamma\|_2^2 + \sqrt{\e}\ \int_0^L \beta_1'(\Gamma)\ \left| \partial_x \gamma \right|^2\ dx \\
& \le & \frac{1}{2}\ \frac{d}{dt} \|\gamma\|_2^2 + \int_0^L \left( \frac{D}{\beta_1'(B)} + \alpha_0(h,H) \right)\ \beta_1'(\Gamma)\ \left| \partial_x \gamma \right|^2\ dx \\
& \le & \e\  \int_0^L \beta_1'(\Gamma)\ \left| \partial_x \gamma \right|^2\ dx + C(\e,T)\ \left( 1 + \left\| \partial_x h \right\|_\infty^2 \right)\ \left( 1 + \|\gamma\|_2^2 \right)\,.
\end{eqnarray*}
Thus, on $\mj\cap [0,T]$ we have
$$
\frac{1}{2}\ \frac{d}{dt} \|\gamma(t)\|_2^2 \le C(\e,T)\ \left( 1 + \left\| \partial_x h(t) \right\|_\infty^2 \right)\ \left( 1 + \|\gamma(t)\|_2^2 \right)\,.
$$
Thanks to \eqref{b118}, the above differential inequality entails that $\|\gamma(t)\|_2^2\le C(\e,T)$ for $t\in\mathcal{J}\cap [0,T]$. As 
$$
|\partial_x\Gamma| = \frac{|\gamma|}{|\sigma'(\Gamma)| \Gamma} \le \frac{|\gamma|}{\sigma_0 \e}
$$
by \eqref{a6} and \eqref{b101}, Lemma~\ref{leb106} follows.
\end{proof}

Gathering \eqref{b101}, \eqref{b117}, \eqref{b128}, and \eqref{b137}, we have thus established that, given $T>0$, there is $C(\e,T)$ such that
$$
\|h(t)\|_{W_2^1} + \|\Gamma(t)\|_{W_2^1} \le C(\e,T)\,, \quad t\in\mathcal{J}\cap [0,T]\,,
$$
from which we deduce that $\mathcal{J}=[0,\infty)$ according to \eqref{globex}. This completes the proof of Theorem~\ref{T1}.

%%%%%%%%%%%%%%%%%%%%%%%%%%%%%%%%%%%%%%%%%%%%%%%%%%%%%%%%%%%%%%%%%%%%
%%%%%%%%%%%%%%%%%%%%%%%%%%%%%%%%%%%%%%%%%%%%%%%%%%%%%%%%%%%%%%%%%%%%
\section{Existence of weak solutions}\label{sec:eows}
%%%%%%%%%%%%%%%%%%%%%%%%%%%%%%%%%%%%%%%%%%%%%%%%%%%%%%%%%%%%%%%%%%%%
%%%%%%%%%%%%%%%%%%%%%%%%%%%%%%%%%%%%%%%%%%%%%%%%%%%%%%%%%%%%%%%%%%%%

Pick $\e\in (0,1/2)$. By Theorem~\ref{T1}, Lemma~\ref{leb101}, and Lemma~\ref{leb103}, there is a unique global strong solution $(h_\e,\Gamma_\e)$ to \eqref{b1}-\eqref{b4} with initial conditions $(h_{0,\e},\Gamma_{0,\e})$ given by \eqref{b13} and satisfying
\begin{eqnarray}
& & h_\e(t,x) \ge \sqrt{\e}\,, \quad \Gamma_\e(t,x) \ge \e\,, \quad (t,x)\in Q_\infty\,, \label{c1} \\
& & \|h_\e(t)\|_1 = \|h_0\|_1+ L\sqrt{\e}\,, \quad \|\Gamma_\e(t)\|_1=\|\Gamma_0\|_1 + L\e\,, \quad t\ge 0\,, \label{c2}
\end{eqnarray}
and
\begin{equation}
\mathcal{L}_\e(t) + \int_0^t \mathcal{D}_\e(s)\ ds \le \mathcal{L}_\e(0)\,, \quad t\ge 0\,,\label{c3} 
\end{equation}
with
\begin{equation}
\mathcal{L}_\e(t) := \int_0^L \left[ \frac{G}{2}\ |h_\e(t,x)|^2 + \phi(\Gamma_\e(t,x)) \right]\ dx\,, \quad t\ge 0\,,\label{c4}
\end{equation}
\begin{eqnarray}
2\mathcal{D}_\e(t) & := &  G\ \|J_{f,\e}(t)\|_2^2 + \|J_{s,\e}(t)\|_2^2 + (\eta_1-\eta)\ \|\sqrt{H_\e(t)} \partial_x\sigma(\Gamma_\e(t))\|_2^2 \label{c5} \\
& & + (1-\eta_1)\ \|\sqrt{h_\e(t)} \partial_x\sigma(\Gamma_\e(t))\|_2^2 + (1-\eta) G\ \|\partial_x\alpha_1(h_\e(t))\|_2^2\nonumber \\
& & + 2D\ \int_0^L \frac{\left| \partial_x\sigma(\Gamma_\e(t)) \right|^2}{\beta_1'(B_\e(t))}\ dx\,, \nonumber
\end{eqnarray}
the functions $J_{f,\e}$, $J_{s,\e}$, $H_\e$, $A_\e$, $B_\e$, $\Sigma_\e$, and $\phi$ being defined in \eqref{b111}, \eqref{b112}, \eqref{b9}, \eqref{b10}, and \eqref{b115}, respectively. We first deduce from \eqref{c3} several estimates which provide us the compactness of $(h_\e,\Gamma_\e)$.

%%%%%%%%%%%%%%%%%%%%%%%%%%%%%%%%%%%%%%%%%%%%%%%%%%%%%%%%%%%%%%%%%%%%
%%%%%%%%%%%%%%%%%%%%%%%%%%%%%%%%%%%%%%%%%%%%%%%%%%%%%%%%%%%%%%%%%%%%
\subsection{Compactness}\label{sec:cp}
%%%%%%%%%%%%%%%%%%%%%%%%%%%%%%%%%%%%%%%%%%%%%%%%%%%%%%%%%%%%%%%%%%%%
%%%%%%%%%%%%%%%%%%%%%%%%%%%%%%%%%%%%%%%%%%%%%%%%%%%%%%%%%%%%%%%%%%%%

Observe that the definition \eqref{b115} of $\phi$ and the property \eqref{a6} of $\sigma$ imply that
\begin{eqnarray}
\sigma_0\ (r\ln{r} - r +1) & \le & \phi(r) \le \sigma_\infty\ (r\ln{r} - r +1)\,, \quad r\ge 0\,, \label{c6} \\
\phi(r) & \le & \max{\{ \phi(0) , \phi(r+1) \}} \le \sigma_\infty + \phi(r+1)\,, \quad r\ge 0\,. \label{c7}
\end{eqnarray}
An easy consequence of \eqref{c4} and \eqref{c7} is that 
\begin{eqnarray*}
\mathcal{L}_\e(0) & = & \int_0^L \left[ G\ \frac{(h_0(x)+\sqrt{\e})^2}{2} + \phi(\Gamma_0(x)+\e) \right]\ dx \\
& \le & \int_0^L \left[ G\ \frac{(h_0(x)+1)^2}{2} + \sigma_\infty + \phi(\Gamma_0(x)+\e +1) \right]\ dx \\
& \le & \int_0^L \left[ G\ \frac{(h_0(x)+1)^2}{2} + \sigma_\infty + \phi(\Gamma_0(x)+2) \right]\ dx \ ,
\end{eqnarray*}
\begin{equation}
\mathcal{L}_\e(0) \le C_4\,. \label{c8}
\end{equation}

This allows us to derive some uniform estimates with respect to $\e$.

\begin{lemma}\label{lec1}
Given $t\ge 0$, we have
\begin{eqnarray}
\|h_\e(t)\|_2 + \int_0^L \Gamma_\e(t,x)\ |\ln{\Gamma_\e(t,x)}|\ dx & \le & C_5\,, \label{c9} \\
\int_0^t \left( \| J_{f,\e}(s) \|_2^2 + \| J_{s,\e}(s) \|_2^2 \right)\ ds & \le & C_5\,, \label{c10} \\
\int_0^t \int_0^L \left[ \left| \partial_x\alpha_1(h_\e) \right|^2 + \left( \sqrt{\e} + h_\e+H_\e + \frac{1}{\beta_1'(B_\e)}  \right)\ \left| \partial_x\sigma(\Gamma_\e) \right|^2  \right]\ dxds & \le & C_5\,, \label{c11} \\
\int_0^t \left( \left\| h_\e^{5/2}(s) \right\|_{W_2^1}^2 + \|h_\e(s)\|_\infty^5 \right)\ ds & \le & C_5\ (1+t)\,. \label{c12}
\end{eqnarray}
\end{lemma}

\begin{proof} Recalling that $r |\ln{r}| \le 1/e$ for $r\in [0,1]$, we deduce from \eqref{c6} that 
\begin{equation}
\sigma_0\ r |\ln{r}| \le \phi(r) + \sigma_0\ r +  \frac{\sigma_0}{e}\,, \quad r\ge 0\,. \label{c13}
\end{equation}
Owing to the nonnegativity of $\mathcal{D}_\e$, it follows from \eqref{c2}, \eqref{c3}, \eqref{c8}, and \eqref{c13} that 
\begin{eqnarray*}
\frac{G}{2}\ \|h_\e(t)\|_2^2 + \sigma_0\ \int_0^L \Gamma_\e(t,x)\ |\ln{\Gamma_\e(t,x)}|\ dx & \le & \mathcal{L}_\e(t) + \sigma_0\ \|\Gamma_\e(t)\|_1 +  \frac{\sigma_0 L}{e} \\
& \le & \mathcal{L}_\e(0) + \sigma_0\ \left( \|\Gamma_0\|_1 + L\e \right) +  \frac{\sigma_0 L}{e} \le C\,,
\end{eqnarray*}
and we obtain \eqref{c9}. Next, \eqref{c10} and \eqref{c11} are straightforward consequences of \eqref{c3}, \eqref{c5}, and \eqref{c8}, since the lower bound \eqref{c1} on $h_\e$ guarantees that $\sqrt{h_\e}\ge \e^{1/4}$. Recalling the definition \eqref{b7} of $\alpha_1$, we infer from \eqref{c11} that
\begin{eqnarray*}
C & \ge & \int_0^t \int_0^L |\alpha_1'(h_\e)|^2\ \left|\partial_x h_\e \right|^2\ dxds\, = \,\frac{G}{3}\ \int_0^t \int_0^L h_\e^3\ \left|\partial_x h_\e \right|^2\ dxds \\
& = & \frac{4G}{75}\ \int_0^t \int_0^L \left|\partial_x\left( h_\e^{5/2} \right) \right|^2\ dxds\,.
\end{eqnarray*}
We next argue as in the proof of \eqref{b126} to establish that
$$
\|h_\e\|_\infty^5 \le C\ \left( \|h_\e\|_1^5 + \left\|\partial_x\left( h_\e^{5/2} \right) \right\|_2^2 \right)\,,
$$
and \eqref{c12} follows from \eqref{c2} and the above two inequalities. 
\end{proof}

We now turn to the compactness properties of $(\Gamma_\e)_\e$ with respect to the space variable and first establish a preliminary result. Recall that $B_\e=\mn_\e (\Gamma_\e)$ with $\mn_\e$ defined in \eqref{b0}.

\begin{lemma}\label{lec2}
Given $t\ge 0$, we have
\begin{equation}
\int_0^L B_\e(t,x)\ |\ln{B_\e(t,x)}|\ dx \le C_6\,. \label{c14}
\end{equation}
\end{lemma}

\begin{proof}
Let $j$ be the convex function defined by $j(r):=r\ln{r}-r$ for $r\ge 0$ with conjugate function $j^*(r):=e^r$, $r\ge 0$. Then
\begin{equation}
r s \le j(r) + j^*(s) \;\;\mbox{ and }\;\; j(r) \le r j'(r)\,, \quad (r,s)\in [0,\infty)^2\,. \label{c15}
\end{equation} 
We infer from the definition of $B_\e$, the convexity of $j$, and \eqref{c15} that 
\begin{eqnarray*}
\int_0^L j(B_\e)\ dx & \le & \int_0^L B_\e\ j'(B_\e)\ dx \le \int_0^L \left( B_\e\ j'(B_\e) + \e^2\ j''(B_\e)\ \left| \partial_x B_\e \right|^2 \right)\ dx \\
& = & \int_0^L j'(B_\e)\ \left( B_\e - \e^2\ \partial_x^2 B_\e \right)\ dx = \int_0^L j'(B_\e)\ \Gamma_\e\ dx \le \int_0^L \left( j(\Gamma_\e) + j^*(j'(B_\e)) \right)\ dx \\
& = & \int_0^L \left( \Gamma_\e\ \ln{\Gamma_\e} - \Gamma_\e + B_\e \right)\ dx\,.
\end{eqnarray*}
Since $\Gamma_\e$ and $B_\e$ are nonnegative and $\|B_\e\|_1 \le \|\Gamma_\e\|_1$ by \eqref{b106} and \eqref{c1}, we end up with
$$
\int_0^L j(B_\e)\ dx \le \int_0^L \Gamma_\e\ \ln{\Gamma_\e} \ dx \le \int_0^L \Gamma_\e\ |\ln{\Gamma_\e}|\ dx\,,
$$
from which we deduce \eqref{c14} with the help of \eqref{c2}, \eqref{c9}, and the elementary inequality $r |\ln{r}| \le r\ln{r} +2/e$, $r\ge 0$.
\end{proof}

We next define the function $\omega_\ell\in\mathcal{C}([0,L])$ by $\omega_\ell(0)=0$ and
\begin{equation}
\omega_\ell (L\delta) := \left[ \ln{(\varrho(\delta))} \right]^{-1/2}\,, \quad \delta\in (0,1]\,, \label{c16}
\end{equation}
where $\varrho(\delta)>1$ denotes the unique solution to
\begin{equation}
\varrho(\delta)\ \ln{(\varrho(\delta))} = \frac{1}{\delta} \;\;\mbox{ for }\;\; \delta\in (0,1]\,. \label{c17}
\end{equation}
Introducing the subset $\mathcal{X}$ of $\mathcal{C}([0,L])$ by
\begin{equation}
\mathcal{X} := \left\{ f\in \mathcal{C}([0,L])\ : \ [f ]_{\mathcal{X}} := \sup_{x\ne y} \frac{|f(x)-f(y)|}{\omega_\ell(|x-y|)} < \infty \right\}\,, \label{c18}
\end{equation}
and noticing that it is a Banach space equipped with the norm
$$
f\mapsto \|f\|_{\mathcal{X}}:=\|f\|_\infty +[ f]_{\mathcal{X}}\, ,
$$
we have the following result:

\begin{lemma}\label{lec3}
Given $t\ge 0$, we have
\begin{equation}
\int_0^t \left( \left\| \partial_x\sigma(\Gamma_\e(s)) \right\|_1^2 + \|\sigma(\Gamma_\e(s))\|_{\mathcal{X}}^2 \right)\ ds \le C_7\,. \label{c19}
\end{equation}
In addition, letting $Q_T:= (0,T)\times (0,L)$ for $T>0$, the family
\begin{equation}
\left( \partial_x\sigma(\Gamma_\e) \right)_\e \;\;\mbox{ is relatively weakly sequentially compact in }\;\; L_1(Q_T)\, \label{c20}
\end{equation}
 and
\begin{equation}
\int_0^T  \|\Gamma_\e(s)\|_{\mathcal{X}}^2\ ds \le C_8(T)\,. \label{c21}
\end{equation}
\end{lemma}

\begin{proof}
For $t\ge 0$ and $0\le y < x \le L$, we note that
\begin{eqnarray}
|\sigma(\Gamma_\e(t,x))-\sigma(\Gamma_\e(t,y)| & \le & \int_y^x \left| \partial_x \sigma(\Gamma_\e(t,z)) \right|\ dz \nonumber \\
& \le & \left( \int_0^L \frac{\left| \partial_x\sigma(\Gamma_\e(t,z)) \right|^2}{\beta_1'(B_\e(t,z))}\ dz \right)^{1/2}\ \left( \int_y^x \beta_1'(B_\e(t,z))\ dz \right)^{1/2}\,. \label{c22}
\end{eqnarray}
We infer from \eqref{a6}, {\eqref{b8},} and \eqref{c14} that, for $R>1$, 
\begin{eqnarray*}
\int_y^x \beta_1'(B_\e(t,z))\ dz & \le & \sigma_\infty\ \int_y^x B_\e(t,z)\ dz \le \sigma_\infty\ \int_y^x \left( \mathbf{1}_{[0,R]}(B_\e(t,z)) + \mathbf{1}_{(R,\infty)}(B_\e(t,z)) \right)\ B_\e(t,z)\ dz \\
& \le & \sigma_\infty R\ |x-y| + \frac{\sigma_\infty}{\ln{R}}\ \int_0^L \mathbf{1}_{(R,\infty)}(B_\e(t,z))\ B_\e(t,z)\ |\ln{B_\e(t,z)}|\ dz \\
& \le & \sigma_\infty R\ |x-y| + \frac{\sigma_\infty C_6}{\ln{R}} \,.
\end{eqnarray*}
Choosing $R=\varrho(|x-y|/L)$ and using \eqref{c17}, we conclude that
$$
\int_y^x \beta_1'(B_\e(t,z))\ dz  \le \frac{C}{\ln{\varrho(|x-y|/L)}} = C\ \omega_\ell(|x-y|)^2\,.
$$
Recalling \eqref{c22}, we have shown that 
$$
|\sigma(\Gamma_\e(t,x))-\sigma(\Gamma_\e(t,y)| \le C\ \omega_\ell(|x-y|)\  \left( \int_0^L \frac{\left| \partial_x\sigma(\Gamma_\e(t,z)) \right|^2}{\beta_1'(B_\e(t,z))}\ dz \right)^{1/2}\,.
$$
Consequently,
$$
[\sigma(\Gamma_\e(t))]_{\mathcal{X}}^2 \le C\ \left( \int_0^L \frac{\left| \partial_x\sigma(\Gamma_\e(t,z)) \right|^2}{\beta_1'(B_\e(t,z))}\ dz \right)\,.
$$
Integrating the above inequality with respect to time and using \eqref{c11} give
\begin{equation}
\int_0^t [\sigma(\Gamma_\e(s))]_{\mathcal{X}}^2\ ds \le C_9\,. \label{c23}
\end{equation}

It also follows from \eqref{a6}, \eqref{b106}, \eqref{c2}, and \eqref{c11} that, for $T>0$, 
\begin{eqnarray*}
\int_0^T \left\| \partial_x\sigma(\Gamma_\e(t)) \right\|_1^2\ dt & \le & \int_0^T \left( \int_0^L \frac{\left| \partial_x\sigma(\Gamma_\e(t,z)) \right|^2}{\beta_1'(B_\e(t,z))}\ dz \right)\ \left( \int_0^L \beta_1'(B_\e(t,z))\ dz \right)\ dt \\
& \le & \sigma_\infty C_5\ \sup_{t\in [0,T]}{\left\{ \|B_\e(t)\|_1 \right\}} \\
& \le & \sigma_\infty C_5\ \sup_{t\in [0,T]}{\left\{ \|\Gamma_\e(t)\|_1 \right\}} \le C\,,
\end{eqnarray*}
which, together with \eqref{a6b}, \eqref{c2}, \eqref{c23}, Poincar\'e's inequality, and the embedding of $W_1^1(0,L)$ in $L_\infty (0,L)$ completes the proof of \eqref{c19}. 

Consider next $T>0$ and a measurable subset $E$ of $Q_T$ with finite measure. Arguing as above, we deduce from \eqref{a6} and \eqref{c11} that, for $R>1$,  
\begin{eqnarray*}
\int_E \left| \partial_x\sigma(\Gamma_\e) \right|\ dxdt & \le & C\ \left( \int_E B_\e\ dxdt \right)^{1/2} \\
& \le & C\ \left[ \int_E \left( \mathbf{1}_{[0,R]}(B_\e) + \mathbf{1}_{(R,\infty)}(B_\e) \right)\ B_\e\ dxdt \right]^{1/2} \\
& \le & C\ \left[ R\ |E| + \frac{1}{\ln{R}}\ \int_{Q_T} \mathbf{1}_{(R,\infty)}(B_\e)\ B_\e\ |\ln{B_\e}|\ dxdt \right]^{1/2} \,.
\end{eqnarray*}
Owing to \eqref{c14}, we conclude that 
$$
\int_E \left| \partial_x\sigma(\Gamma_\e) \right|\ dxdt \le C\ \left( R\ |E| + \frac{1}{\ln{R}} \right)^{1/2}\,,
$$
and thus
$$
\limsup_{\delta\to 0} \sup_{\e,|E|\le\delta}{\left\{ \int_E \left| \partial_x\sigma(\Gamma_\e) \right|\ dxdt \right\}} \le \left( \frac{C}{\ln{R}} \right)^{1/2}\,.
$$
Letting $R\to\infty$ entails \eqref{c20} by the Dunford-Pettis theorem. 

Finally, by \eqref{a6}, we have
$$
\sigma_0\ |r-s| = \sigma_0\ (r-s) \le \sigma(s)-\sigma(r) = |\sigma(r)-\sigma(s)|\,, \quad r\ge s\ge 0\,,
$$
so that 
$$
[\Gamma_\e(t)]_{\mathcal{X}} \le \frac{ [\sigma(\Gamma_\e(t))]_{\mathcal{X}}}{\sigma_0} \;\;\mbox{ and }\;\; \|\Gamma_\e(t)\|_\infty \le \frac{\sigma(0) + \|\sigma(\Gamma_\e(t))\|_\infty}{\sigma_0}\,, \quad t\ge 0\,,
$$
and \eqref{c21} follows at once from \eqref{c19}.
\end{proof}

The next result deals with the time compactness of $(h_\e)$ and $(\Gamma_\e)$. 

\begin{lemma}\label{lec4}
Let $T>0$. Then
\begin{eqnarray}
& & \left( \partial_t h_\e \right)_\e \;\;\mbox{ is bounded in }\;\; L_{5/4}(0,T;W_2^1(0,L)')\,, \label{c23bbb}\\
& & \left( \partial_t \Gamma_\e \right)_\e \;\;\mbox{ is bounded in }\;\; L_{18/17}(0,T;W_{18/17}^1(0,L)')\,. \label{c24}
\end{eqnarray}
\end{lemma}

\begin{proof}
By \eqref{b1} and \eqref{b111}, we have
$$
\partial_t h_\e = \partial_x\left( \sqrt{a_1(h_\e)}\ J_{f,\e} \right)\,.
$$ 
As $(J_{f,\e})_\e$ is bounded in $L_2(Q_T)$ by \eqref{c10} and $(\sqrt{a_1(h_\e)})_\e$ is bounded in $L_{10/3}(0,T;L_\infty(0,L))$ by \eqref{c12}, the family $\left( \sqrt{a_1(h_\e)}\ J_{f,\e} \right)_\e$ is bounded in $L_{5/4}(0,T;L_2(0,L))$ and \eqref{c23bbb} readily follows from this property.

Next, owing to \eqref{b8} and \eqref{b112}, equation \eqref{b2} also reads
\begin{equation}
\partial_t \Gamma_\e = \partial_x \left( - \sqrt{\alpha_0(h_\e,H_\e)}\ \Gamma_\e\ J_{s,\e} - D\ \frac{\partial_x \sigma(\Gamma_\e)}{\sqrt{\beta_1'(B_\e)}}\ \frac{\Gamma_\e}{\sqrt{\beta_1'(B_\e)}} \right)\,. \label{c24b}
\end{equation} 
On the one hand, it follows from H\"older's inequality, \eqref{b103a}, \eqref{c2}, and \eqref{c9} that
\begin{eqnarray*}
\int_0^T \left\| \sqrt{\alpha_0(h_\e,H_\e)}\ \Gamma_\e\ J_{s,\e} \right\|_{8/7}^{16/15}\ dt & \le & \int_0^T \left\| \alpha_0(h_\e,H_\e) \right\|_2^{8/15}\ \|\Gamma_\e\|_8^{16/15}\  \|J_{s,\e}\|_2^{16/15}\ dt \\
& \le & \int_0^T \| h_\e\|_2^{8/15}\ \|\Gamma_\e\|_\infty^{14/15}\ \|\Gamma_\e\|_1^{2/15}\ \|J_{s,\e}\|_2^{16/15}\ dt \\
& \le & C\ \left( \int_0^T \|\Gamma_\e\|_\infty^2\ dt \right)^{7/15}\ \left( \int_0^T \|J_{s,\e}\|_2^2\ dt \right)^{8/15} \,.
\end{eqnarray*}
We then deduce from \eqref{c10} and \eqref{c21} that 
\begin{equation}
\left( \sqrt{\alpha_0(h_\e,H_\e)}\ \Gamma_\e\ J_{s,\e} \right)_\e \;\;\mbox{ is bounded in }\;\; L_{16/15}(0,T;L_{8/7}(0,L))\,. \label{c25}
\end{equation}
On the other hand, it follows from \eqref{a6}, \eqref{b9}, and \eqref{b106} that
\begin{equation}
\left| \frac{\Gamma_\e}{\sqrt{\beta_1'(B_\e)}} \right| \le \left| \frac{\Gamma_\e}{\sqrt{\sigma_0\ B_\e}} \right| \le \frac{B_\e + \e^2\ \left| \partial_x^2 B_\e \right|}{\sqrt{\sigma_0\ B_\e}} \le \sqrt{\frac{B_\e}{\sigma_0}} + \frac{\e^{3/2}}{\sqrt{\sigma_0}}\ \left| \partial_x^2 B_\e \right|\,.\label{c26}
\end{equation}
Since
$$
\left\| \e^2 \partial_x^2 B_\e \right\|_3 \le \|\Gamma_\e\|_3 + \|B_\e\|_3 \le 2\ \|\Gamma_\e\|_3
$$
by \eqref{b9} and \eqref{b106}, we deduce from \eqref{b107}, \eqref{c2}, and H\"older's inequality that 
\begin{eqnarray*}
\left\| \e^{3/2} \partial_x^2 B_\e \right\|_{9/4}^{9/4} & \le & \e^{27/8}\ \left\| \partial_x^2 B_\e \right\|_2^{3/2}\ \left\| \partial_x^2 B_\e \right\|_3^{3/4} \le \e^{3/8}\ \left( \frac{\left\| \partial_x \Gamma_\e \right\|_2^2}{2} \right)^{3/4}\ \left( 2\ \|\Gamma_\e\|_3 \right)^{3/4} \\
& \le & C\ \left( \e^{1/2}\ \left\| \partial_x \Gamma_\e \right\|_2^2 \right)^{3/4}\ \|\Gamma_\e\|_\infty^{1/2}\ \|\Gamma_\e\|_1^{1/4} \\
& \le & C\ \left( \int_0^L \e^{1/2}\ \left| \partial_x \Gamma_\e \right|^2\ dx + \|\Gamma_\e\|_\infty^2 \right)\,.
\end{eqnarray*}
Thanks to \eqref{c11} and \eqref{c21}, the above inequality implies that 
$$
\int_0^T \left\| \e^{3/2} \partial_x^2 B_\e \right\|_{9/4}^{9/4}\ dt \le C(T)\,.
$$
As 
$$
 \|\sqrt{B_\e}\|_{9/4}^{9/4} = \|B_\e\|_{9/8}^{9/8}\le \|\Gamma_\e\|_{9/8}^{9/8} \le \|\Gamma_\e\|_\infty^{1/8}\ \|\Gamma_\e\|_1 \le C\ \|\Gamma_\e\|_\infty^{1/8}
$$
by \eqref{b106} and \eqref{c2}, we infer from \eqref{c21} that $(\sqrt{B_\e})_\e$ is bounded in $L_{9/4}(Q_T)$ and conclude that the right-hand side of \eqref{c26} is bounded in $L_{9/4}(Q_T)$. Consequently,
\begin{equation}
\left( \frac{\Gamma_\e}{\sqrt{\beta_1'(B_\e)}} \right)_\e \;\;\mbox{ is bounded in }\;\; L_{9/4}(Q_T)\,. \label{c27}
\end{equation}
Recalling that $\left( \partial_x\sigma(\Gamma_\e)/\sqrt{\beta_1'(B_\e)} \right)_\e$ is bounded in $L_2(Q_T)$ by \eqref{c11}, we end up with 
\begin{equation}
\left( \frac{\partial_x \sigma(\Gamma_\e)\ \Gamma_\e}{\beta_1'(B_\e)} \right)_\e \;\;\mbox{ is bounded in }\;\; L_{18/17}(Q_T)\,. \label{c28}
\end{equation}
The claim \eqref{c24} is now a straightforward consequence of \eqref{c24b}, \eqref{c25}, and \eqref{c28}. 
\end{proof}

Thanks to the previous analysis, we have the following compactness properties on the families $(h_\e)_\e$ and $(\Gamma_\e)_\e$.

\begin{lemma}\label{lec5}
For each $T>0$ and $\vartheta\in [0,1/5)$, 
\begin{eqnarray}
(h_\e)_\e \;\;\mbox{ is relatively compact in }\;\; L_5(0,T;\mathcal{C}^\vartheta([0,L]))\,, \label{c29} \\
(\Gamma_\e)_\e \;\;\mbox{ is relatively compact in }\;\; L_2(0,T;\mathcal{C}([0,L]))\,. \label{c30}
\end{eqnarray}
\end{lemma}

\begin{proof}
For $(t,x,y)\in (0,\infty]\times (0,L)^2$, we have 
$$
|h_\e(t,x)-h_\e(t,y)| \le \left| h_\e(t,x)^{5/2} - h_\e(t,y)^{5/2} \right|^{2/5} \le |x-y|^{1/5}\ \left\| \partial_x \left( h_\e^{5/2} \right)(t) \right\|_2^{2/5}\,,
$$
and we infer from \eqref{c12} that 
\begin{equation}
(h_\e)_\e \;\;\mbox{ is bounded in }\;\; L_5(0,T;\mathcal{C}^{1/5}([0,L]))\,. \label{c31}
\end{equation}
By the Arzel\`a-Ascoli theorem, $\mathcal{C}^{1/5}([0,L])$ is compactly embedded in $\mathcal{C}^\vartheta([0,L])$ for all $\vartheta\in [0,1/5)$ and it follows from \eqref{c23bbb}, \eqref{c31}, and \cite[Corollary~4]{Si87} that \eqref{c29} holds true.

Similarly, $\mathcal{X}$ is compactly embedded in $\mathcal{C}([0,L])$ (since $\omega_\ell(\delta)\to 0$ as $\delta\to 0$) and we infer from \eqref{c21}, \eqref{c24}, and \cite[Corollary~4]{Si87} that \eqref{c30} holds true.
\end{proof}

%%%%%%%%%%%%%%%%%%%%%%%%%%%%%%%%%%%%%%%%%%%%%%%%%%%%%%%%%%%%%%%%%%%%
%%%%%%%%%%%%%%%%%%%%%%%%%%%%%%%%%%%%%%%%%%%%%%%%%%%%%%%%%%%%%%%%%%%%
\subsection{Convergence}\label{sec:cv}
%%%%%%%%%%%%%%%%%%%%%%%%%%%%%%%%%%%%%%%%%%%%%%%%%%%%%%%%%%%%%%%%%%%%
%%%%%%%%%%%%%%%%%%%%%%%%%%%%%%%%%%%%%%%%%%%%%%%%%%%%%%%%%%%%%%%%%%%%

According to \eqref{c10}, \eqref{c11}, \eqref{c20}, and Lemma~\ref{lec5}, there are functions $h$, $g_1$, $\Gamma$, $g$, $\overline{J_f}$, and $\overline{J_s}$ and a sequence $(\e_k)_k$, $\e_k\to 0$, such that, for all $T>0$ and $\vartheta\in [0,1/5)$, 
\begin{eqnarray}
& & h_{\e_k} \longrightarrow h \;\;\mbox{ in }\;\; L_5(0,T;\mathcal{C}^\vartheta([0,L])) \;\;\mbox{ and a.e. in }\;\; Q_T\,, \label{c101} \\
& & \partial_x \alpha_1(h_{\e_k}) \rightharpoonup g_1 \;\;\mbox{ in }\;\; L_2(Q_T)\,, \label{c102} \\
& & \Gamma_{\e_k} \longrightarrow \Gamma \;\;\mbox{ in }\;\; L_2(0,T;\mathcal{C}([0,L])) \;\;\mbox{ and a.e. in }\;\; Q_T\,, \label{c103} \\
& & \partial_x \sigma(\Gamma_{\e_k}) \rightharpoonup g \;\;\mbox{ in }\;\; L_1(Q_T)\,, \label{c104} \\
& & J_{f,\e_k} \rightharpoonup \overline{J_f} \;\;\mbox{ in }\;\; L_2(Q_T)\,, \label{c105} \\
& & J_{s,\e_k} \rightharpoonup \overline{J_s} \;\;\mbox{ in }\;\; L_2(Q_T)\,. \label{c106}
\end{eqnarray}
An obvious consequence of \eqref{c1}, \eqref{c2}, \eqref{c101}, and \eqref{c103} is that 
\begin{eqnarray}
h\ge 0\,, \quad \Gamma\ge 0\,, \quad g_1=\partial_x\alpha_1(h)\,, \quad g=\partial_x\sigma(\Gamma)\,, \label{c104b} \\
\|h(t)\|_1=\|h_0\|_1\,, \quad \|\Gamma(t)\|_1=\|\Gamma_0\|_1\,, \quad t\ge 0\,. \label{c104c}
\end{eqnarray}

The next step is to investigate the convergence of $(H_{\e_k})_k$, $(\Sigma_{\e_k})_k$, $(A_{\e_k})_k$, and $(B_{\e_k})_k$ in the light of \eqref{c101} and \eqref{c103}. For that purpose, we need the following preliminary results.

\begin{lemma}\label{lec6}
Consider $s_1\in (0,1)$. There is $C_9=C_9 (s_1)>0$ such that, for all $\e>0$,
\begin{equation}
\|\mathcal{N}_\e (w)\|_{\mathcal{C}^{s_1}} \le C_9\ \|w \|_{\mathcal{C}^{s_1}}\,, \quad w\in\mathcal{C}^{s_1}([0,L])\,, \label{alpha0}
\end{equation}
the operator $\mathcal{N}_\e$ being defined in \eqref{b0}.
\end{lemma}

\begin{proof}
Interpreting 
$$
\mn_\e=\frac{1}{\e^2}\left(\frac{1}{\e^2}-\partial_x^2\right)^{-1}
$$
as a resolvent on $\mc^{s_1}$ of the negative Laplacian subject to homogeneous Neumann boundary conditions and noting that the latter has zero spectral bound, the assertion readily follows from \cite[Cor.~3.1.32]{Lunardi} and \cite[Def.~2.0.1]{Lunardi}.
\end{proof}

\begin{lemma}\label{lec7}
Consider $0<s_0<s_1<1$. There are numbers $\vartheta\in (0,1)$, $C_{10}>0$, and $p\ge 2$, all depending on $s_0$ and $s_1$, such that
\begin{equation}
\|w\|_{\mathcal{C}^{s_0}} \le C_{10}\ \|w \|_{\mathcal{C}^{s_1}}^\vartheta\ \|w\|_p^{1-\vartheta}\,, \quad w\in\mathcal{C}^{s_1}([0,L])\,. \label{alpha1}
\end{equation}
\end{lemma}

\begin{proof}
Let $s_0<\sigma_0< \sigma_1<s_1$ and let $ \nu\in (0,1)$ and $p\ge 2$ be such that $s_1=\sigma_1+2 \nu$ and $s_0\le \sigma_0-1/p$. On the one hand, since $\mathcal{C}^{ \nu}([0,L])$ is continuously embedded in $L_p(0,L)$ and $\mathcal{C}^{1+\nu}([0,L])$ is continuously embedded in $W_p^1(0,L)$, interpolation theory guarantees that $\left( \mathcal{C}^{ \nu}([0,L]) , \mathcal{C}^{1+ \nu}([0,L]) \right)_{\sigma_1+ \nu,\infty}$ is continuously embedded in $\left( L_p(0,L) , W_p^1(0,L) \right)_{\sigma_1+ \nu,\infty}$, where $\left( \cdot , \cdot \right)_{\sigma_1+ \nu,\infty}$ denotes the real interpolation method. Since 
$$
\left( \mathcal{C}^{ \nu}([0,L]) , \mathcal{C}^{1+ \nu}([0,L]) \right)_{\sigma_1+ \nu,\infty} = \mathcal{C}^{\sigma_1+2 \nu}([0,L]) \;\;\mbox{ and }\;\; \left( L_p(0,L) , W_p^1(0,L) \right)_{\sigma_1+ \nu,\infty} = B_{p,\infty}^{\sigma_1+ \nu}(0,L)
$$
by \cite[ (5.1), (5.21), (5.22) ]{AmannTeubner93} and $B_{p,\infty}^{\sigma_1+ \nu}(0,L)$ is continuously embedded in $B_{p,1}^{\sigma_1}(0,L)$ which is itself continuously embedded in $B_{p,p}^{\sigma_1}(0,L)=W_p^{\sigma_1}(0,L)$ by \cite[(5.3), (5.5)]{AmannTeubner93}, we have shown that 
\begin{equation}
\| w \|_{W_p^{\sigma_1}} \le C(s_1,\sigma_1,p)\ \|w\|_{\mathcal{C}^{s_1}}\,, \quad w\in\mathcal{C}^{s_1}([0,L])\,. \label{alpha2}
\end{equation}
On the other hand, we have
$$
\| w \|_{W_p^{\sigma_0}} \le C(\sigma_0,\sigma_1,p)\ \| w \|_{W_p^{\sigma_1}}^{\sigma_0/\sigma_1}\ \|w\|_p^{(\sigma_1-\sigma_0)/\sigma_1}\,, \quad w\in W_p^{\sigma_1}(0,L)\,,
$$
while the choice of $p$ implies that $W_p^{\sigma_0}(0,L)$ is continuously embedded in $\mathcal{C}^{s_0}([0,L])$ by \cite[(5.1), (5.3), (5.5), (5.8)]{AmannTeubner93}. Consequently,
\begin{equation}
\| w \|_{\mathcal{C}^{s_0}} \le C(s_0,\sigma_0,p)\ \| w \|_{W_p^{\sigma_0}}  \le C(s_0,\sigma_0,\sigma_1,p)\ \| w \|_{W_p^{\sigma_1}}^{\sigma_0/\sigma_1}\ \|w\|_p^{(\sigma_1-\sigma_0)/\sigma_1}\,, \quad w\in W_p^{\sigma_1}(0,L)\,. \label{alpha3}
\end{equation}
Combining \eqref{alpha2} and \eqref{alpha3} gives \eqref{alpha1} with $\vartheta=\sigma_0/\sigma_1$.
\end{proof}

A useful consequence of \eqref{c101}, Lemma~\ref{lec6}, and Lemma~\ref{lec7} is the convergence of $(H_{\e_k})_k$.

\begin{lemma}\label{lem8} Given $T>0$, we have
\begin{equation}
H_{\e_k} \longrightarrow h \;\;\mbox{ in }\;\; L_5(0,T;\mathcal{C}([0,L])) \;\;\mbox{ and a.e. in }\;\; Q_T\,. \label{zz2}
\end{equation}
\end{lemma}

\begin{proof}
We first claim that 
\begin{equation}
H_{\e_k} \longrightarrow h \;\;\mbox{ in }\;\; L_2(Q_T) \;\;\mbox{ and a.e. in }\;\; Q_T\,. \label{c106d}
\end{equation}
Indeed, it follows from \eqref{b103a}, \eqref{c9}, and \eqref{c12} that 
\begin{equation}
\|H_\e(t)\|_2 \le C_5 \;\;\mbox{ and }\;\; \int_0^t \|H_\e(s)\|_\infty^5\ ds \le C_5\ (1+t) \label{c106b}
\end{equation}
for $t\ge 0$. We may thus assume (after possibly extracting a further subsequence) that 
\begin{equation}
H_{\e_k} \rightharpoonup h \;\;\mbox{ in }\;\; L_2(Q_T) \label{c106c}
\end{equation}
for all $T>0$. It also follows from  the definition \eqref{b9} of $H_\e$ that 
$$
\|H_\e(t)\|_2^2 + \e^2\ \left\| \partial_x H_\e(t) \right\|_2^2 = \int_0^L H_\e(t,x)\ h_\e(t,x)\ dx\,, \quad t\ge 0\,.
$$
Consequently, given $T>0$, we deduce from \eqref{c106b} that 
\begin{eqnarray*}
& & \int_0^T \|H_\e(t) - h(t) \|_2^2\ dt \\ 
& \le & \int_0^T \left[ \|H_\e(t)\|_2^2 + \e^2\ \left\| \partial_x H_\e(t) \right\|_2^2 + \|h(t)\|_2^2 - 2\ \int_0^L H_\e(t,x)\ h(t,x)\ dx \right]\ dt \\
& \le & \int_0^T \int_0^L \left[ H_\e(t,x)\ (h_\e(t,x) -h(t,x)) + h(t,x)\ (h(t,x) - H_\e(t,x)) \right]\ dxdt \\
& \le & C_5 T^{1/2}\ \left( \int_0^T \| h_\e(t)-h(t)\|_2^2\ dt \right)^{1/2} + \int_0^T \int_0^L h(t,x)\ (h(t,x) - H_\e(t,x))\ dxdt\,.
\end{eqnarray*}
We then take $\e=\e_k$ in the above inequality and pass to the limit as $k\to\infty$ with the help of \eqref{c101} and \eqref{c106c} to complete the proof of \eqref{c106d}, extracting possibly a further subsequence to obtain the convergence almost everywhere. 

Now, fix $s_0\in (0,1/5)$ and let $\vartheta\in (0,1)$ and $p\ge 2$ be given by Lemma~\ref{lec7} with $s_1=1/5$. By \eqref{c31}, \eqref{c106b}, Lemma~\ref{lec6} (with $s_1=1/5$), Lemma~\ref{lec7}, and H\"older's inequality, we have
\begin{eqnarray*}
& & \int_0^T \|H_{\e_k}(t)-h(t)\|_{\mathcal{C}^{s_0}}^5\ dt \\ 
& \le & C(s_0)\ \int_0^T \|H_{\e_k}(t)-h(t)\|_{\mathcal{C}^{1/5}}^{5\vartheta}\ \|H_{\e_k}(t)-h(t)\|_p^{5(1-\vartheta)}\ dt \\
& \le & C(s_0)\ \int_0^T \left( \|H_{\e_k}(t)\|_{\mathcal{C}^{1/5}}^{5\vartheta} + \|h(t)\|_{\mathcal{C}^{1/5}}^{5\vartheta} \right)\ \|H_{\e_k}(t)-h(t)\|_\infty^{5(p-2)(1-\vartheta)/p}\ \|H_{\e_k}(t)-h(t)\|_2^{10(1-\vartheta)/p}\ dt \\
& \le & C(s_0)\ \int_0^T \|h(t)\|_{\mathcal{C}^{1/5}}^{5\vartheta}\ \left( \|H_{\e_k}(t) \|_\infty + \|h(t)\|_\infty \right)^{5(p-2)(1-\vartheta)/p}\ \|H_{\e_k}(t)-h(t)\|_2^{10(1-\vartheta)/p}\ dt \\
& \le & C(s_0,T)\ \left( \int_0^T \left( \|H_{\e_k}(t)\|_\infty + \|h(t)\|_\infty \right)^{5(p-2)/p}\ \|H_{\e_k}(t)-h(t)\|_2^{10/p}\ dt \right)^{1-\vartheta} \\
& \le & C(s_0,T)\ \left( \int_0^T \|H_{\e_k}(t)-h(t)\|_2^5\ dt \right)^{2(1-\vartheta)/p} \\
& \le & C(s_0,T)\ \left( \int_0^T \|H_{\e_k}(t)-h(t)\|_2^2\ dt \right)^{2(1-\vartheta)/p}\ ,
\end{eqnarray*}
 where we have used \eqref{b103a} and \eqref{c9} to obtain the last inequality.
The convergence \eqref{zz2} then follows by \eqref{c106d} thanks to the continuous embedding of $\mathcal{C}^{s_0}([0,L])$ in $\mathcal{C}([0,L])$. 
\end{proof}

The last result of this section is devoted to $(A_\e)_\e$, $(B_\e)_\e$, and $(\Sigma_\e)_\e$.

\begin{lemma}\label{lec8} For $T>0$,  we have
\begin{eqnarray}
& & A_{\e_k} \longrightarrow \alpha_1(h) \;\;\mbox{ in }\;\; L_2(Q_T) \;\;\mbox{ and a.e. in }\;\; Q_T\,, \label{c106g}\\ 
& & \partial_x A_{\e_k} \rightharpoonup \partial_x \alpha_1(h) \;\;\mbox{ in }\;\; L_2(Q_T)\,, \label{c133}\\
& & B_{\e_k} \longrightarrow \Gamma \;\;\mbox{ in }\;\; L_2(Q_T) \;\;\mbox{ and a.e. in }\;\; Q_T\,, \label{c106e}\\
& & \Sigma_{\e_k} \longrightarrow \sigma(\Gamma) \;\;\mbox{ in }\;\; L_2(Q_T) \;\;\mbox{ and a.e. in }\;\; Q_T\,, \label{c106f}
\end{eqnarray}
after possibly extracting a further subsequence.
\end{lemma}

\begin{proof}
The proofs of \eqref{c106g} and \eqref{c106e} are similar to that of \eqref{c106d}, the necessary bounds stemming from \eqref{c11} and \eqref{c21}.

Concerning $(\Sigma_\e)_\e$, we infer from \eqref{b10} that 
$$
\|\Sigma_\e\|_2^2 + \e^2\ \left\| \sqrt{H_\e}\ \partial_x \Sigma_\e \right\|_2^2 = \int_0^L \sigma(\Gamma_\e)\ \Sigma_\e\ dx \le \frac{1}{2}\ \|\sigma(\Gamma_\e)\|_2^2 + \frac{1}{2}\ \|\Sigma_\e\|_2^2\,,
$$
so that $(\Sigma_\e)_\e$ and $\left( \e \sqrt{H_\e}\ \partial_x \Sigma_\e \right)_\e$ are bounded in $L_2(Q_T)$ by \eqref{c19}. Consequently, $(\Sigma_\e)_\e$ is weakly relatively compact in $L_2(Q_T)$ while $\left( \e^2 \sqrt{H_\e}\ \partial_x \Sigma_\e \right)_\e$ converges to zero in $L_2(Q_T)$. These information along with \eqref{b10} and \eqref{c103} allow us to conclude that we have, after possibly extracting a further subsequence, the weak convergence in $L_2(Q_T)$ of $(\Sigma_{\e_k})_k$ to $\sigma(\Gamma)$. We then argue as in the proof of \eqref{c106d} to complete the proof of \eqref{c106f}. 

Finally, owing to \eqref{b104} and \eqref{c11}, $\left( \partial_x A_\e \right)_\e$ is bounded in $L_2(Q_T)$ from which \eqref{c133} follows by \eqref{c106g} after possibly extracting a further subsequence.
\end{proof}

%%%%%%%%%%%%%%%%%%%%%%%%%%%%%%%%%%%%%%%%%%%%%%%%%%%%%%%%%%%%%%%%%%%%
%%%%%%%%%%%%%%%%%%%%%%%%%%%%%%%%%%%%%%%%%%%%%%%%%%%%%%%%%%%%%%%%%%%%
\subsection{Passing to the limit in \eqref{b1}}\label{sec:pttlib1}
%%%%%%%%%%%%%%%%%%%%%%%%%%%%%%%%%%%%%%%%%%%%%%%%%%%%%%%%%%%%%%%%%%%%
%%%%%%%%%%%%%%%%%%%%%%%%%%%%%%%%%%%%%%%%%%%%%%%%%%%%%%%%%%%%%%%%%%%%

Observing that \eqref{b1} also reads $\partial_t h_\e = - \partial_x \left( \sqrt{a_1(h_\e)}\ J_{f,\e} \right)$, we have
\begin{equation}
\frac{d}{dt} \int_0^L h_\e\ \psi\ dx = \int_0^L \partial_x \psi\ \sqrt{a_1(h_\e)}\ J_{f,\e}\ dx \label{c107}
\end{equation}
for all $\psi\in W_\infty^1(0,L)$. 

Now, it follows from \eqref{c9} and \eqref{c101} that $\left( \sqrt{a_1(h_{\e_k})} \right)_k$ converges toward $\sqrt{a_1(h)}$ in $L_2(Q_T)$ for $T>0$. Combining this convergence with \eqref{c105} yields that $\left( \sqrt{a_1(h_{\e_k})}\ J_{f,\e_k} \right)_k$ converges weakly toward $\sqrt{a_1(h)}\ \overline{J_f}$ in $L_1(Q_T)$. We may then pass to the limit in \eqref{c107} and find that 
\begin{equation}
\frac{d}{dt} \int_0^L h\ \psi\ dx = \int_0^L \partial_x \psi\ \sqrt{a_1(h)}\ \overline{J_f}\ dx \label{c109}
\end{equation}
for all $\psi\in W_\infty^1(0,L)$. 

%%%%%%%%%%%%%%%%%%%%%%%%%%%%%%%%%%%%%%%%%%%%%%%%%%%%%%%%%%%%%%%%%%%%
%%%%%%%%%%%%%%%%%%%%%%%%%%%%%%%%%%%%%%%%%%%%%%%%%%%%%%%%%%%%%%%%%%%%
\subsection{Passing to the limit in \eqref{b2}}\label{sec:pttlib2}
%%%%%%%%%%%%%%%%%%%%%%%%%%%%%%%%%%%%%%%%%%%%%%%%%%%%%%%%%%%%%%%%%%%%
%%%%%%%%%%%%%%%%%%%%%%%%%%%%%%%%%%%%%%%%%%%%%%%%%%%%%%%%%%%%%%%%%%%%

We note that \eqref{b2} also reads 
\begin{equation}
\partial_t \Gamma_\e = - \partial_x \left( D\ \frac{\Gamma_\e}{\beta_1'(B_\e)}\ \partial_x\sigma(\Gamma_\e) + \sqrt{\alpha_0(h_\e,H_\e)}\ \Gamma_\e\ J_{s,\e} \right)\,. \label{c110}
\end{equation}

Let $T>0$. We first identify the limit of the second term in the right-hand side of \eqref{c110}. It follows from \eqref{c2} and \eqref{c21} that 
\begin{eqnarray*}
\int_0^T \left\| \left( \sqrt{\alpha_0(h_\e,H_\e)} - \sqrt{h} \right)\ \Gamma_\e \right\|_2^2\ dt & \le & \int_0^T \int_0^L \left| \alpha_0(h_\e,H_\e) - h \right|\ \Gamma_\e^2\ dxdt \\
& \le & \int_0^T \|\alpha_0(h_\e,H_\e)-h\|_\infty\ \|\Gamma_\e\|_\infty\ \|\Gamma_\e\|_1\ dt \\
& \le & C\ \int_0^T \left( \|H_\e-h\|_\infty + \|h_\e-h\|_\infty \right)\ \|\Gamma_\e\|_\infty\ dt \\
& \le & C(T)\ \left( \int_0^T \left( \|H_\e-h\|_\infty^2 + \|h_\e-h\|_\infty^2 \right)\ dt \right)^{1/2}\,,
\end{eqnarray*}
whence
\begin{equation}
\lim_{k\to\infty} \int_0^T \left\| \left( \sqrt{\alpha_0(h_{\e_k},H_{\e_k})} - \sqrt{h} \right)\ \Gamma_{\e_k} \right\|_2^2\ dt = 0 \label{c112}
\end{equation}
by \eqref{c101} and \eqref{zz2}. In addition, owing to \eqref{c104c}, we have
$$
\int_0^T \left\| \sqrt{h}\ \left( \Gamma_\e-\Gamma \right) \right\|_2^2\ dt \le \int_0^T \|h\|_1\ \|\Gamma_\e-\Gamma\|_\infty^2\ dt \le C\ \int_0^T \|\Gamma_\e-\Gamma\|_\infty^2\ dt\,,
$$
so that
\begin{equation}
\lim_{k\to\infty} \int_0^T \left\| \sqrt{h}\ \left( \Gamma_{\e_k}-\Gamma \right) \right\|_2^2\ dt = 0 \label{c113}
\end{equation}
by \eqref{c103}. Gathering \eqref{c112} and \eqref{c113}, we have established that 
$$
\sqrt{\alpha_0(h_{\e_k},H_{\e_k})}\ \Gamma_{\e_k} \longrightarrow \sqrt{h}\ \Gamma \;\;\mbox{ in }\;\; L_2(Q_T)\,,
$$
which, together with \eqref{c106}, implies that 
\begin{equation}
\sqrt{\alpha_0(h_{\e_k},H_{\e_k})}\ \Gamma_{\e_k}\ J_{s,\e_k} \rightharpoonup \sqrt{h}\ \Gamma\ \overline{J_s} \;\;\mbox{ in }\;\; L_1(Q_T)\,.\label{c114}
\end{equation}
We now turn to the first term of the right-hand side of \eqref{c110} and use \eqref{b9} to obtain
\begin{equation}
\frac{\Gamma_\e}{\beta_1'(B_\e)}\ \partial_x\sigma(\Gamma_\e) = \frac{\partial_x\sigma(\Gamma_\e)}{|\sigma'(B_\e)|} - \e^2\ \frac{\partial_x^2 B_\e}{\sqrt{\beta_1'(B_\e)}}\ \frac{\partial_x\sigma(\Gamma_\e)}{\sqrt{\beta_1'(B_\e)}}\,. \label{c115}
\end{equation}
On the one hand, it follows from \eqref{a6}, \eqref{b106}, \eqref{b107}, and repeated use of \eqref{c11} that 
\begin{eqnarray*}
\e^2\ \int_0^T \int_0^L \left| \frac{\partial_x^2 B_\e}{\sqrt{\beta_1'(B_\e)}}\ \frac{\partial_x\sigma(\Gamma_\e)}{\sqrt{\beta_1'(B_\e)}} \right|\ dxdt & \le & C_5 \e^2\ \left( \int_0^T \int_0^L \frac{\left| \partial_x^2 B_\e \right|^2}{\sigma_0 B_\e}\ dxdt \right)^{1/2} \\
& \le & C \e^2\ \left( \int_0^T \frac{\left\| \partial_x^2 B_\e \right\|_2^2}{\e}\ dt \right)^{1/2} \\
& \le & C \e^{3/2}\ \left( \int_0^T \frac{\left\| \partial_x \Gamma_\e \right\|_2^2}{2 \e^2}\ dt \right)^{1/2} \\
& \le & C \e^{1/2}\ \left( \int_0^T \frac{\left\| \partial_x \sigma(\Gamma_\e) \right\|_2^2}{\sigma_0^2}\ dt \right)^{1/2} \\
& \le & C \e^{1/4}\ \left( \int_0^T \sqrt{\e}\ \left\| \partial_x \sigma(\Gamma_\e) \right\|_2^2\ dt \right)^{1/2} \\
& \le & C \e^{1/4}\,,
\end{eqnarray*}
so that
\begin{equation}
\lim_{\e\to 0}\ \e^2\ \int_0^T \left\| \frac{\partial_x^2 B_\e}{\sqrt{\beta_1'(B_\e)}}\ \frac{\partial_x\sigma(\Gamma_\e)}{\sqrt{\beta_1'(B_\e)}} \right\|_1\ dt =0\,. \label{c116}
\end{equation}
On the other hand, we have
$$
\frac{1}{|\sigma'(B_{\e_k})|} \le \frac{1}{\sigma_0} \;\;\mbox{ and }\;\; \frac{1}{|\sigma'(B_{\e_k})|} \longrightarrow \frac{1}{|\sigma'(\Gamma)|} \;\;\mbox{ a.e. in }\;\; Q_T
$$
by \eqref{a6} and \eqref{c106e}. Recalling that $\left( \partial_x \sigma(\Gamma_{\e_k}) \right)_k$ converges weakly toward $\partial_x\sigma(\Gamma)$ in $L_1(Q_T)$ by \eqref{c104} and \eqref{c104b}, Lemma~\ref{leap1} (see the appendix) ensures that 
\begin{equation}
\frac{\partial_x\sigma(\Gamma_{\e_k})}{|\sigma'(B_{\e_k})|} \rightharpoonup \frac{\partial_x\sigma(\Gamma)}{|\sigma'(\Gamma)|} \;\;\mbox{ in }\;\; L_1(Q_T)\,. \label{c117}
\end{equation}
Furthermore, as $\sigma$ is a Lipschitz continuous diffeomorphism with a Lipschitz continuous inverse and $\partial_x\sigma(\Gamma)\in L_1(Q_T)$, we have also $\partial_x \Gamma\in L_1(Q_T)$ with $\partial_x\Gamma = \partial_x\sigma(\Gamma)/\sigma'(\Gamma)$. Consequently, we may pass to the limit in \eqref{c110} and deduce from \eqref{c114}, \eqref{c115}, \eqref{c116}, and \eqref{c117} that 
\begin{equation}
\frac{d}{dt} \int_0^L \Gamma\ \psi\ dx = \int_0^L \partial_x\psi\ \left[ -D\ \partial_x\Gamma + \sqrt{h}\ \Gamma\ \overline{J_s} \right]\ dx \label{c118}
\end{equation}
for all $\psi\in W_\infty^1(0,L)$.

%%%%%%%%%%%%%%%%%%%%%%%%%%%%%%%%%%%%%%%%%%%%%%%%%%%%%%%%%%%%%%%%%%%%
%%%%%%%%%%%%%%%%%%%%%%%%%%%%%%%%%%%%%%%%%%%%%%%%%%%%%%%%%%%%%%%%%%%%
\subsection{Identifying $\overline{J_f}$}\label{sec:ijf}
%%%%%%%%%%%%%%%%%%%%%%%%%%%%%%%%%%%%%%%%%%%%%%%%%%%%%%%%%%%%%%%%%%%%
%%%%%%%%%%%%%%%%%%%%%%%%%%%%%%%%%%%%%%%%%%%%%%%%%%%%%%%%%%%%%%%%%%%%

Recalling \eqref{c102}, \eqref{c104b} and the formula
$$
J_{f,\e} = -\partial_x \alpha_1(h_\e) + \frac{a_{2,\e}(h_\e)\ \sqrt{H_\e}}{\sqrt{h_\e a_1(h_\e)}}\ \partial_x \Sigma_\e\,,
$$
the key toward the identification of the limit of $J_{f,\e}$ is the behavior as $\e\to 0$ of the term involving $\partial_x \Sigma_\e$. 
At this point, we observe that \eqref{b109} and \eqref{c11} guarantee that $\left( \sqrt{H_\e}\ \partial_x \Sigma_\e\right)_\e$ is bounded in $L_2(Q_T)$ for all $T>0$, so that this quantity has weak cluster points in $L_2(Q_T)$. However, nothing is known so far on $\left( \partial_x \Sigma_\e \right)_\e$ and it is yet unclear whether these cluster points can be determined in terms of $h$ and $\sigma(\Gamma)$. The aim of the next result is to remedy to this fact.

\begin{lemma}\label{lec9}
Given $T>0$, the family $\left( \partial_x \Sigma_\e \right)_\e$ is bounded in $L_2(0,T;L_1(0,L))$ and relatively weakly sequentially compact in $L_1(Q_T)$.
\end{lemma}

In order not to delay further the identification of $\overline{J_f}$, we postpone the proof of Lemma~\ref{lec9}. Let $T>0$. Recalling \eqref{c106f}, we deduce from Lemma~\ref{lec9} that, after possibly extracting a further subsequence, we have
\begin{equation}
\partial_x \Sigma_{\e_k} \rightharpoonup \partial_x \sigma(\Gamma) \;\;\mbox{ in }\;\; L_1(Q_T)\,. \label{c121}
\end{equation}
On the one hand, since $\left( H_{\e_k}/(1+H_{\e_k}) \right)_k$ is bounded due to the positivity of $H_{\e_k}$ and converges a.e. to $h/(1+h)$ by \eqref{c106d}, we use once more Lemma~\ref{leap1} to conclude that 
$$
\sqrt{\frac{H_{\e_k}}{1+H_{\e_k}}}\ \partial_x \Sigma_{\e_k} \rightharpoonup \sqrt{\frac{h}{1+h}}\ \partial_x \sigma(\Gamma) \;\;\mbox{ in }\;\; L_1(Q_T)\,. 
$$
On the other hand, it follows from \eqref{b109}, \eqref{c11}, and the positivity of $H_{\e_k}$ that 
$$
\left( \sqrt{\frac{H_\e}{1+H_\e}}\ \partial_x \Sigma_\e \right)_\e \;\;\mbox{ is bounded in }\;\; L_2(Q_T)\,. 
$$
Combining these two properties implies, after possibly extracting a further subsequence, that
\begin{equation}
\sqrt{\frac{H_{\e_k}}{1+H_{\e_k}}}\ \partial_x \Sigma_{\e_k} \rightharpoonup \sqrt{\frac{h}{1+h}}\ \partial_x \sigma(\Gamma) \;\;\mbox{ in }\;\; L_2(Q_T)\,.\label{c120}
\end{equation}
We next observe that Lebesgue's dominated convergence theorem and \eqref{c101} ensure that
$$
\frac{a_{2,\e_k}(h_{\e_k})}{\sqrt{h_{\e_k} a_1(h_{\e_k})}} \longrightarrow \sqrt{\frac{3}{4G}} \;\;\mbox{ in }\;\; L_4(Q_T)\,.
$$
Since $\left( \sqrt{1+H_{\e_k}} \right)_k$ converges toward $\sqrt{1+h}$ in $L_4(Q_T)$ by \eqref{zz2}, we end up with
\begin{equation}
\frac{a_{2,\e_k}(h_{\e_k}) \sqrt{1+H_{\e_k}}}{\sqrt{h_{\e_k} a_1(h_{\e_k})}} \longrightarrow \sqrt{\frac{3 (1+h)}{4G}} \;\;\mbox{ in }\;\; L_2(Q_T)\,. \label{c123}
\end{equation}

We then infer from \eqref{c102}, \eqref{c105}, \eqref{c104b}, \eqref{c120}, and \eqref{c123} that 
\begin{equation}
\overline{J_f} = -\partial_x \alpha_1(h) + \sqrt{\frac{3h}{4G}}\ \partial_x \sigma(\Gamma)  =j_f\,. \label{c124}
\end{equation}
 In particular, thanks to \eqref{c102}, \eqref{c105}, and \eqref{c104b},
\begin{equation}
\sqrt{h}\ \partial_x \sigma(\Gamma) \in L_2(Q_T)\,. \label{barbapapa}
\end{equation}

\begin{proof}[Proof of Lemma~\ref{lec9}]
We put $\xi_\e:= \partial_x \Sigma_\e$. Let $\Theta\in\mathcal{C}^2(\RR)$ be a nonnegative and convex function  satisfying $\Theta(0)=0$ and define the function $\Theta_1$ by $\Theta_1(0)=0$ and $\Theta_1'(r)=r \Theta''(r)$, $r\in\RR$. Since $\xi_\e$ solves
$$
\xi_\e - \e^2\ \partial_x^2 \left( H_\e\ \xi_\e \right) = \partial_x \sigma(\Gamma_\e) \;\;\mbox{ in }\;\; (0,L) \;\;\mbox{ with }\;\; \xi_\e(0)=\xi_\e(L)=0\,,
$$
by \eqref{b10}, we have
\begin{eqnarray*}
\int_0^L \Theta'(\xi_\e)\ \xi_\e\ dx & = & - \e^2\ \int_0^L \Theta''(\xi_\e)\ \partial_x\xi_\e\ \partial_x \left( H_\e\ \xi_\e \right)\ dx  + \int_0^L \Theta'(\xi_\e)\ \partial_x \sigma(\Gamma_\e)\ dx \\
& \le & - \e^2\ \int_0^L \partial_x \Theta_1(\xi_\e)\ \partial_x H_\e\ dx  + \int_0^L \Theta'(\xi_\e)\ \partial_x \sigma(\Gamma_\e)\ dx \,.
\end{eqnarray*}
On the one hand, performing an integration by parts and using \eqref{b9} and the nonnegativity of $\Theta_1$ and $h_\e$ give
$$
- \e^2\ \int_0^L \partial_x \Theta_1(\xi_\e)\ \partial_x H_\e\ dx = \int_0^L \Theta_1(\xi_\e)\ (H_\e-h_\e)\ dx \le \int_0^L \Theta_1(\xi_\e)\ H_\e\ dx\,.
$$
On the other hand, it follows from the convexity of $\Theta$ that
$$
\int_0^L \Theta'(\xi_\e)\ \partial_x \sigma(\Gamma_\e)\ dx \le \int_0^L \left[ \Theta'(\xi_\e)\ \xi_\e - \Theta(\xi_\e) + \Theta\left( \partial_x \sigma(\Gamma_\e) \right) \right]\ dx\,.
$$
Consequently, gathering the previous three inequalities we obtain
\begin{equation}
\int_0^L \Theta(\xi_\e)\ dx \le \int_0^L \Theta_1(\xi_\e)\ H_\e\ dx  + \int_0^L \Theta\left( \partial_x \sigma(\Gamma_\e) \right)\ dx \,. \label{spirou}
\end{equation}

We first use \eqref{spirou} to obtain an $L_1$-bound on $(\xi_\e)_\e$. For $\delta\in (0,1)$ and $r\in\RR$, define $\Phi_\delta(r):= \sqrt{r^2+\delta^2}-\delta$. It is a nonnegative and convex function vanishing at zero and we infer from \eqref{spirou} with $\Theta=\Phi_\delta$ that
$$
\int_0^L \Phi_\delta(\xi_\e)\ dx \le \delta^2\ \int_0^L \left( \frac{1}{\delta} - \frac{1}{\sqrt{\xi_\e^2+\delta^2}} \right)\ H_\e\ dx +  \int_0^L \Phi_\delta\left( \partial_x \sigma(\Gamma_\e) \right)\ dx \,.
$$
We then pass to the limit as $\delta\to 0$ and conclude that $\|\xi_\e\|_1\le \left\| \partial_x \sigma(\Gamma_\e) \right\|_1$. Integrating this inequality with respect to time and using \eqref{c19} then give that
\begin{equation}
\left( \xi_\e \right)_\e \;\;\mbox{ is bounded in }\;\;  L_2(0,T;L_1(0,L))\,. \label{fantasio}
\end{equation}

To improve \eqref{fantasio}, we need a refined version of the de la Vall\'ee-Poussin theorem (recalled in Lemma~\ref{leap2} below) which asserts that the weak compactness \eqref{c20} of $\left( \partial_x \sigma(\Gamma_\e) \right)_\e$ in $L_1(Q_T)$ implies the existence of a nonnegative and even convex function $\Psi\in\mathcal{C}^2(\RR)$ such that $\Psi(0)=0$, $\Psi'$ is concave on $[0,\infty)$,  
\begin{equation}
K(T) := \sup_{\e\in (0,1)} \int_0^T \int_0^L \Psi\left( \partial_x \sigma(\Gamma_\e) \right)\ dxdt < \infty \;\;\mbox{ and }\;\; \lim_{r\to\infty} \frac{\Psi(r)}{r} =  \infty\,. \label{spip}
\end{equation}
Then $ 0 \le \Psi''(r)\le \Psi''(0)$ for $r\in\RR$ and it follows from \eqref{spirou} with $\Theta=\Psi$ that
$$
\int_0^L \Psi(\xi_\e)\ dx \le \frac{\Psi''(0)}{2}\ \int_0^L \xi_\e^2\ H_\e\ dx +  \int_0^L \Psi\left( \partial_x \sigma(\Gamma_\e) \right)\ dx\,.
$$
Integrating over $(0,T)$ and using \eqref{b109}, \eqref{c11}, and \eqref{spip}, we end up with 
$$
\int_0^T \int_0^L \Psi(\xi_\e)\ dxdt \le \frac{C_5 \Psi''(0)}{2} + K(T)\,.
$$
Since $\Psi$ is even and superlinear at infinity by \eqref{spip}, the previous bound  implies the uniform integrability of $(\xi_\e)_\e$ in $L_1(Q_T)$ and the Dunford-Pettis theorem entails the expected result.
\end{proof}

%%%%%%%%%%%%%%%%%%%%%%%%%%%%%%%%%%%%%%%%%%%%%%%%%%%%%%%%%%%%%%%%%%%%
%%%%%%%%%%%%%%%%%%%%%%%%%%%%%%%%%%%%%%%%%%%%%%%%%%%%%%%%%%%%%%%%%%%%
\subsection{Identifying $\overline{J_s}$}\label{sec:ijs}
%%%%%%%%%%%%%%%%%%%%%%%%%%%%%%%%%%%%%%%%%%%%%%%%%%%%%%%%%%%%%%%%%%%%
%%%%%%%%%%%%%%%%%%%%%%%%%%%%%%%%%%%%%%%%%%%%%%%%%%%%%%%%%%%%%%%%%%%%

We first recall that 
$$
J_{s,\e} = \sqrt{\alpha_0(h_\e,H_\e)}\ \partial_x\sigma(\Gamma_\e) - G\ \frac{a_{2,\e}(h_\e)}{\sqrt{h_\e a_1(h_\e)}}\ \frac{b_2(\Gamma_\e)}{\Gamma_\e}\ \partial_x A_\e\,.
$$
Let $T>0$. For $\psi\in L_\infty(Q_T)$ and $\delta\in (0,1)$, we have
\begin{equation}
\int_0^T \int_0^L \psi\ \left( \sqrt{\alpha_0(h_{\e_k},H_{\e_k})}\ \partial_x\sigma(\Gamma_{\e_k}) - \sqrt{h}\ \partial_x\sigma(\Gamma) \right)\ dxdt = I_{1,k} + I_{2,k}(\delta) + I_{3,k}(\delta)\,, \label{c126b}
\end{equation}
with
\begin{eqnarray*}
I_{1,k} & := & \int_0^T \int_0^L \psi\ \left( \sqrt{\alpha_0(h_{\e_k},H_{\e_k})} - \sqrt{h} \right)\ \partial_x\sigma(\Gamma_{\e_k})\ dxdt \,,\\
I_{2,k}(\delta) & := & \int_0^T \int_0^L \psi\ \left( \sqrt{h} - \sqrt{\frac{h}{1+\delta h}} \right)\ (\partial_x\sigma(\Gamma_{\e_k}) -\partial_x\sigma(\Gamma))\ dxdt \,, \\ 
I_{3,k}(\delta) & := & \int_0^T \int_0^L \psi\ \sqrt{\frac{h}{1+\delta h}}\ \left( \partial_x\sigma(\Gamma_{\e_k}) - \partial_x\sigma(\Gamma) \right)\ dxdt \,.
\end{eqnarray*}
We infer from \eqref{c19} and H\"older's inequality that 
\begin{eqnarray*}
\left| I_{1,k} \right| & \le & \|\psi\|_{L_\infty(Q_T)}\ \int_0^T \left\| \sqrt{\alpha_0(h_{\e_k},H_{\e_k})} - \sqrt{h} \right\|_\infty\ \left\| \partial_x\sigma(\Gamma_\e) \right\|_1\ dt \\
& \le & \sqrt{C_7}\ \|\psi\|_{L_\infty(Q_T)}\ \left( \int_0^T \left\| \sqrt{\alpha_0(h_{\e_k},H_{\e_k})} - \sqrt{h} \right\|_\infty^2\ dt \right)^{1/2} \,,
\end{eqnarray*}
whence, by \eqref{c101} and \eqref{zz2},
\begin{equation}
\lim_{k\to\infty} I_{1,k}=0\,. \label{c127b}
\end{equation}
Using again \eqref{c19} along with \eqref{c104c} and \eqref{barbapapa}, we find
\begin{eqnarray*}
\left| I_{2,k}(\delta) \right| & \le & \|\psi\|_{L_\infty(Q_T)}\ \int_0^T \int_0^L \sqrt{\delta}\ h\ \ ( \left| \partial_x\sigma(\Gamma_\e) \right| + \left| \partial_x\sigma(\Gamma) \right|)\ dxdt \\
& \le & \sqrt{\delta}\ \|\psi\|_{L_\infty(Q_T)}\ \int_0^T  \left[ \|h\|_\infty\ \left\| \partial_x\sigma(\Gamma_\e) \right\|_1 + \| h\|_1^{1/2}\ \left\| \sqrt{h}\ \partial_x\sigma(\Gamma) \right\|_2 \right]\ dt \\
& \le &  \sqrt{\delta}\ \|\psi\|_{L_\infty(Q_T)}\ \left[ C_7\ \left( \int_0^T \|h(t)\|_\infty^2\ dt \right)^{1/2} + C(T) \right]\,.
\end{eqnarray*}
Since $h$ belongs to $L_5(0,T;L_\infty(0,L))$, we conclude that 
\begin{equation}
\sup_{k\ge 1}{\left| I_{2,k}(\delta) \right|} \le C(T)\ \sqrt{\delta}\ \|\psi\|_{L_\infty(Q_T)}\,. \label{c128}
\end{equation}
Finally, owing to \eqref{c104}, \eqref{c104b}, and the boundedness of $h/(1+\delta h)$, we have
\begin{equation}
\lim_{k\to \infty} I_{3,k}(\delta) = 0\,. \label{c129b}
\end{equation}
It then follows from \eqref{c126b}, \eqref{c127b}, \eqref{c128}, and \eqref{c129b} that 
$$
\limsup_{k\to \infty} \left| \int_0^T \int_0^L \psi\ \left( \sqrt{\alpha_0(h_{\e_k},H_{\e_k})}\ \partial_x\sigma(\Gamma_{\e_k}) - \sqrt{h}\ \partial_x\sigma(\Gamma) \right)\ dxdt \right| \le C(T)\ \sqrt{\delta}\ \|\psi\|_{L_\infty(Q_T)}\,.
$$
Since $\delta$ is arbitrary in $(0,1)$, we may let $\delta\to 0$ in the previous inequality and realize that
\begin{equation}
\sqrt{\alpha_0(h_{\e_k},H_{\e_k})}\ \partial_x\sigma(\Gamma_{\e_k}) \rightharpoonup \sqrt{h}\ \partial_x\sigma(\Gamma) \;\;\mbox{ in }\;\; L_1(Q_T)\,. \label{c130}
\end{equation}

Next, since
$$
0 \le \frac{a_{2,\e}(h_\e)}{\sqrt{h_\e a_1(h_\e)}} \le \sqrt{\frac{3}{4G}} \;\;\mbox{ and }\;\; 0\le \frac{b_2(\Gamma_\e)}{\Gamma_\e} \le 1\,,
$$
we readily infer from \eqref{c101}, \eqref{c103}, and Lebesgue's dominated convergence theorem that 
$$
\frac{a_{2,\e_k}(h_{\e_k})}{\sqrt{h_{\e_k} a_1(h_{\e_k})}}\ \frac{b_2(\Gamma_{\e_k})}{\Gamma_{\e_k}} \longrightarrow  \sqrt{\frac{3}{4G}} \;\;\mbox{ in }\;\; L_2(Q_T)\,.
$$
Combining this property with \eqref{c133} yields
\begin{equation}
\frac{a_{2,\e_k}(h_{\e_k})}{\sqrt{h_{\e_k} a_1(h_{\e_k})}}\ \frac{b_2(\Gamma_{\e_k})}{\Gamma_{\e_k}}\ \partial_x A_{\e_k} \rightharpoonup \sqrt{\frac{3}{4G}}\ \partial_x\alpha_1(h) \;\;\mbox{ in }\;\; L_1(Q_T)\,. \label{c130b}
\end{equation}

Thanks to \eqref{c106}, \eqref{c130}, and \eqref{c130b}, we have identified $\overline{J_s}$:
\begin{equation}
\overline{J_s} = \sqrt{h}\ \partial_x\sigma(\Gamma) - \sqrt{\frac{3G}{4}}\ \partial_x\alpha_1(h)  =j_s\,. \label{c136}
\end{equation}

%%%%%%%%%%%%%%%%%%%%%%%%%%%%%%%%%%%%%%%%%%%%%%%%%%%%%%%%%%%%%%%%%%%%
%%%%%%%%%%%%%%%%%%%%%%%%%%%%%%%%%%%%%%%%%%%%%%%%%%%%%%%%%%%%%%%%%%%%
\subsection{The energy inequality}\label{sec:tei}
%%%%%%%%%%%%%%%%%%%%%%%%%%%%%%%%%%%%%%%%%%%%%%%%%%%%%%%%%%%%%%%%%%%%
%%%%%%%%%%%%%%%%%%%%%%%%%%%%%%%%%%%%%%%%%%%%%%%%%%%%%%%%%%%%%%%%%%%%

Let $T>0$. Since $\overline{J_f}=j_f$ and $\overline{J_s}=j_s$ by \eqref{c124} and \eqref{c136}, we infer from \eqref{c102}, \eqref{c105}, \eqref{c106}, and \eqref{c104b} that
\begin{eqnarray*}
\int_0^T \int_0^L j_f^2\ dxdt & \le & \liminf_{k\to\infty} \int_0^T \int_0^L J_{f,\e_k}^2\ dxdt\,, \\
\int_0^T \int_0^L j_s^2\ dxdt & \le & \liminf_{k\to\infty} \int_0^T \int_0^L J_{s,\e_k}^2\ dxdt\,, \\
\int_0^T \left\| \partial_x\alpha_1(h)\right\|_2^2\ dt & \le & \liminf_{k\to\infty} \int_0^T \left\| \partial_x\alpha_1(h_{\e_k})\right\|_2^2\ dt \,.
\end{eqnarray*} 
We next set $z_{n,\e} := \sqrt{\min{\{h_\e, n \}}} \partial_x\sigma(\Gamma_\e)$ for $n\ge 1$ and $\e\in(0,1)$ and observe that 
\begin{equation}
\label{t1}
\int_0^T \|z_{n,\e}\|_2^2\ dt \le \int_0^T \left\| \sqrt{h_\e}\ \partial_x\sigma(\Gamma_\e) \right\|_2^2\ dt \le C_5
\end{equation} 
by \eqref{c11}. Fix $n\ge 1$. As $\left( \sqrt{\min{\{ h_{\e_k} , n \}}} \right)_k$ is bounded in $L_\infty(Q_T)$ and converges a.e. toward $\sqrt{\min{\{h , n \}}}$ by \eqref{c101}, it follows from \eqref{c104}, \eqref{c104b}, and Lemma~\ref{leap1} that $(z_{n,\e_k})_k$ converges weakly in $L_1(Q_T)$ toward $\sqrt{\min{\{h,n\}}} \partial_x\sigma(\Gamma)$ and also in $L_2(Q_T)$ according to \eqref{t1} (after possibly extracting a further subsequence). We then infer from \eqref{t1} that
$$
\int_0^T \int_0^L \min{\{h,n\}}\ |\partial_x \sigma(\Gamma)|^2\ dxdt \le \liminf_{k\to\infty} \int_0^T \left\| \sqrt{h_{\e_k}}\ \partial_x\sigma(\Gamma_{\e_k}) \right\|_2^2\ dt\,.
$$
Since the right-hand side of the above inequality does not depend on $n$, Fatou's lemma leads us to 
$$
\int_0^T \int_0^L h\ |\partial_x \sigma(\Gamma)|^2\ dxdt \le \liminf_{k\to\infty} \int_0^T \left\| \sqrt{h_{\e_k}}\ \partial_x\sigma(\Gamma_{\e_k}) \right\|_2^2\ dt\,.
$$
A similar argument ensures that
$$
\int_0^T \int_0^L h\ |\partial_x \sigma(\Gamma)|^2\ dxdt \le \liminf_{k\to\infty} \int_0^T \left\| \sqrt{H_{\e_k}}\ \partial_x\sigma(\Gamma_{\e_k}) \right\|_2^2\ dt\,.
$$
Finally, for $\delta\in (0,1)$ and $\e\in (0,1)$, we define
$$
\zeta_{\delta,\e} := \frac{\partial_x \sigma(\Gamma_\e)}{\sqrt{(B_\e+\delta) |\sigma'(B_\e)|}}
$$
and deduce from \eqref{b8} and \eqref{c11} that
\begin{equation}
\label{t2}
\int_0^T \|\zeta_{\delta,\e}\|_2^2\ dt \le \int_0^T \int_0^L \frac{|\partial_x \sigma(\Gamma_\e)|^2}{\beta_1'(B_\e)} \le C_5\,.
\end{equation}
Owing to \eqref{a6} and \eqref{c106e}, $\left( \left( (B_{\e_k}+\delta) |\sigma'(B_{\e_k})| \right)^{-1/2} \right)_k$ is bounded in $L_\infty(Q_T)$ (by $1/\sqrt{\delta\sigma_0}$) and converges a.e. toward $\left( (\Gamma+\delta) |\sigma'(\Gamma)| \right)^{-1/2}$ in $Q_T$. Using once more \eqref{c104}, \eqref{c104b}, and Lemma~\ref{leap1}, we conclude that $\left( \zeta_{\delta,\e_k} \right)_k$ converges weakly toward $\partial_x\sigma(\Gamma)/\sqrt{(\Gamma+\delta) |\sigma'(\Gamma)|}$ in $L_1(Q_T)$ and also in $L_2(Q_T)$ by virtue of \eqref{t2}. Taking the liminf in \eqref{t2} gives
$$
\int_0^T \int_0^L \frac{|\partial_x \sigma(\Gamma)|^2}{(\Gamma+\delta) |\sigma'(\Gamma)|}\ dxdt \le \liminf_{k\to\infty} \int_0^T \int_0^L \frac{|\partial_x \sigma(\Gamma_{\e_k})|^2}{\beta_1'(B_{\e_k})}\ dxdt\,.
$$
Using again \eqref{a6}, we further deduce
$$
4 \sigma_0\ \int_0^T \|\partial_x\sqrt{\Gamma+\delta}\|_2^2\ dt \le \liminf_{k\to\infty} \int_0^T \int_0^L \frac{|\partial_x \sigma(\Gamma_{\e_k})|^2}{\beta_1'(B_{\e_k})}\ dxdt\,.
$$
The above inequality readily implies that $\sqrt{\Gamma}$ belongs to $L_2(0,T;W_2^1(0,L))$ and 
$$
4 \sigma_0\ \int_0^T \|\partial_x\sqrt{\Gamma}\|_2^2\ dt \le \liminf_{k\to\infty} \int_0^T \int_0^L \frac{|\partial_x \sigma(\Gamma_{\e_k})|^2}{\beta_1'(B_{\e_k})}\ dxdt\,.
$$
Collecting the above information and taking into account \eqref{c3}-\eqref{c5} together with \eqref{c101}, \eqref{c103} we conclude the energy inequality \eqref{energineq} since $\eta=3/4$. 

%%%%%%%%%%%%%%%%%%%%%%%%%%%%%%%%%%%%%%%%%%%%%%%%%%%%%%%%%%%%%%%%%%%%
%%%%%%%%%%%%%%%%%%%%%%%%%%%%%%%%%%%%%%%%%%%%%%%%%%%%%%%%%%%%%%%%%%%%
\appendix
\section{Auxiliary results}
%%%%%%%%%%%%%%%%%%%%%%%%%%%%%%%%%%%%%%%%%%%%%%%%%%%%%%%%%%%%%%%%%%%%
%%%%%%%%%%%%%%%%%%%%%%%%%%%%%%%%%%%%%%%%%%%%%%%%%%%%%%%%%%%%%%%%%%%%

We first recall a classical consequence of the Dunford-Pettis and Egorov theorems, see, e.g. \cite[Lemma~A.2]{LaurencotMischler} for a proof. 
 
\begin{lemma}\label{leap1}
Let $U$ be an open bounded subset of $\RR^m$, $m\ge 1$, and consider two sequences $(v_n)_n$ in $L_1(U)$ and $(w_n)_n$ in $L_\infty(U)$ and functions $v\in L_1(U)$ and $w\in L_\infty(U)$ such that
$$
v_n \rightharpoonup v \;\;\mbox{ in }\;\; L_1(U)\,,
$$
$$
\vert w_n(x)\vert \le C \;\;\mbox{ and }\;\; \lim_{n\to\infty} w_n(x) = w(x) \;\;\mbox{ a.e. in }\;\; U 
$$
for some $C>0$. Then
$$
\lim_{n\to \infty} \int_U \vert v_n\vert\ \vert w_n - w\vert\ dx = 0 \;\;\mbox{ and }\;\; v_n\ w_n \rightharpoonup v\ w \;\;\mbox{ in }\;\; L_1(U)\,.
$$
\end{lemma}

We next recall a refined version of the de la  Vall\'ee-Poussin theorem \cite{Le77}. 

\begin{lemma} \label{leap2}
Let $U$ be an open bounded subset of $\RR^m$, $m\ge 1$, and $\mathcal{F}$ a subset of $L_1(U)$. The following two statements are equivalent:
\begin{itemize}
\item[(i)] $\mathcal{F}$ is uniformly integrable, that is, $\mathcal{F}$ is a bounded subset of $L_1(U)$ such that 
$$
\lim_{c\to \infty}\ \sup_{f\in \mathcal{F}}\ \int_{\{ |f|\ge c \}} |f|\ dx = 0\,.
$$
\item[(ii)] $\mathcal{F}$ is a bounded subset of $L_1(U)$ and there exists a convex function $\Phi\in \mathcal{C}^\infty([0,\infty))$ 
such that $\Phi(0) = \Phi'(0) = 0$, $\Phi'$ is a concave function, 
$$
\lim_{r\to \infty} \frac{\Phi(r)}{r} = \lim_{r\to \infty} \Phi'(r) = \infty, \;\;\mbox{ and }\;\; \sup_{f\in \mathcal{F}} \int_U \Phi\left({ | f| }\right)\ dx <\infty\,.
$$
\end{itemize}
\end{lemma}

A proof of Lemma~\ref{leap2} may also be found in \cite{DM75} and \cite[Theorem~I.1.2]{RR91} but without the concavity condition on the first derivative of $\Phi$. Since the sequential weak compactness in $L^1(U)$ implies (and is actually equivalent to, thanks to the boundedness of $U$) the uniform integrability by the Dunford-Pettis theorem, the existence of the function $\Psi$ in the proof of Lemma~\ref{lec9} indeed follows from \eqref{c20} and Lemma~\ref{leap2}. 

%%%%%%%%%%%%%%%%%%%%%%%%%%%%%%%%%%%%%%%%%%%%%%%%%%%%%%%%%%%%%%%%%%%%
%%%%%%%%%%%%%%%%%%%%%%%%%%%%%%%%%%%%%%%%%%%%%%%%%%%%%%%%%%%%%%%%%%%%
\section*{Acknowledgments}
%%%%%%%%%%%%%%%%%%%%%%%%%%%%%%%%%%%%%%%%%%%%%%%%%%%%%%%%%%%%%%%%%%%%
%%%%%%%%%%%%%%%%%%%%%%%%%%%%%%%%%%%%%%%%%%%%%%%%%%%%%%%%%%%%%%%%%%%%
This research was initiated during a visit of J.E. and Ch.W. at the Universit\'e de Toulouse. They greatly acknowledge the hospitality.
Part of the work was done while M.H. and Ph.L. enjoyed the hospitality of the Institut f\"ur Angewandte Mathematik at the Leibniz Universit\"at Hannover.

%%%%%%%%%%%%%%%%%%%%%%%%%%%%%%%%%%%%%%%%%%%%%%%%%%%%%%%%%%%%%%%%%%%%
%%%%%%%%%%%%%%%%%%%%%%%%%%%%%%%%%%%%%%%%%%%%%%%%%%%%%%%%%%%%%%%%%%%%

%%%%%%%%%%%%%%%%%%%%%%%%%%%%%%%%%%%%%%%%%%%%%%%%%%%%%%%%%%%%%%%%%%%%
%%%%%%%%%%%%%%%%%%%%%%%%%%%%%%%%%%%%%%%%%%%%%%%%%%%%%%%%%%%%%%%%%%%%

\end{document}